\documentclass{elsart}
\usepackage[english]{babel}
\usepackage{amsmath,amssymb,enumerate,bbm,xcolor}
\usepackage{mathabx}
\usepackage[pagebackref,hypertexnames=false, colorlinks,
citecolor=red, linkcolor=red]{hyperref}
\usepackage[colorlinks,citecolor=red,pagebackref,hypertexnames=false]{hyperref}
\usepackage[backrefs]{amsrefs}

\newenvironment{proof}{ {\it Proof.} }{\hfill{\it{QED}}\medskip}
\catcode`\<=\active \def<{
\fontencoding{T1}\selectfont\symbol{60}\fontencoding{\encodingdefault}}
\catcode`\>=\active \def>{
\fontencoding{T1}\selectfont\symbol{62}\fontencoding{\encodingdefault}}
\catcode`\|=\active \def|{
\fontencoding{T1}\selectfont\symbol{124}\fontencoding{\encodingdefault}}

\newcommand{\ave}[1]{\langle #1 \rangle}
\newcommand{\pair}[2]{\langle #1,#2 \rangle}

\newtheorem{lemma}{Lemma}
\newtheorem{theorem}{Theorem}
\newtheorem{definition}{Definition}
\newtheorem{remark}{Remark}
\newtheorem{corollary}{Corollary}

\if@mathematic
   \def\vec#1{\ensuremath{\mathchoice
                     {\mbox{\boldmath$\displaystyle\mathbf{#1}$}}
                     {\mbox{\boldmath$\textstyle\mathbf{#1}$}}
                     {\mbox{\boldmath$\scriptstyle\mathbf{#1}$}}
                     {\mbox{\boldmath$\scriptscriptstyle\mathbf{#1}$}}}}
\else
   \def\vec#1{\ensuremath{\mathchoice
                     {\mbox{\boldmath$\displaystyle#1$}}
                     {\mbox{\boldmath$\textstyle#1$}}
                     {\mbox{\boldmath$\scriptstyle#1$}}
                     {\mbox{\boldmath$\scriptscriptstyle#1$}}}}
\fi

\begin{document}

\begin{frontmatter}
  \title{Higher order Journ\'e commutators and characterizations of multi-parameter BMO}
  
   \author[Brown]{Yumeng Ou\thanksref{YO}}
  \author[IMT]{Stefanie Petermichl\thanksref{SP}\thanksref{fn1}}
  \ead{stefanie.petermichl@gmail.com}
  \ead[url]{http://math.univ-toulouse.fr/{$\tilde{\hspace{0.5em}}$}petermic}
  \author[IMB]{Elizabeth Strouse}

   \thanks[YO]{Research supported in part by NSF-DMS 0901139.}
  \thanks[SP]{Research supported in part by ANR-12-BS01-0013-02. The author is a member of IUF.}
  \thanks[fn1]{Correponding author, Tel:+33 5 61 55 76 59, Fax: +33 5 61 55 83 85}
 
 \address[Brown]{Department of mathematics, Brown University, 151 Thayer Street, Providence RI 02912, USA}
  \address[IMT]{Institut de Math{\'e}matiques de Toulouse, Universit{\'e} Paul Sabatier, Toulouse, France}
   \address[IMB]{Institut de Math{\'e}matiques de Bordeaux, 351 cours de la Lib\'eration, F-33405 Talence, France}
  
  \begin{abstract}
  
    We characterize $L^p$ boundedness of iterated commutators of multiplication by a symbol function and tensor products of Riesz and Hilbert transforms. 
    We obtain a two-sided norm estimate that shows that such operators are bounded on $L^p$ if and only if the symbol belongs to the appropriate multi-parameter BMO class. We extend our results to a much more intricate situation; commutators of multiplication by a symbol function and paraproduct-free Journ\'e operators. We show that the boundedness of these commutators is also determined by the inclusion of their symbol function in the same multi-parameter BMO class. In this sense the tensor products of Riesz transforms are a representative testing class for Journ\'e operators.

  Previous results in this direction do not apply to tensor products
and only to Journ\'e operators which can be reduced to Calder\'on-Zygmund operators. Upper norm estimate of Journ\'e commutators are new even in the case of no iterations. 
Lower norm estimates for iterated commutators only existed when no tensor products were present. In the case of one dimension, lower estimates were known for products of two Hilbert transforms, and without iterations. New methods using Journ\'e operators are developed  to obtain these lower norm estimates in the multi-parameter real variable setting.

      \end{abstract}


\end{frontmatter}

\section{Introduction} \label{s1}

As dual of the Hardy space $H^1$, the classical space of functions of bounded mean oscillation, BMO, arises naturally in many endpoint results in analysis, partial differential equations and probability. When entering a setting with several free parameters, a large variety of spaces are encountered, some of which lose the feature of mean oscillation itself. We are interested in characterizations of multi-parameter BMO spaces through boundedness of commutators.

A classical result of Nehari {\cite{Ne}}  shows that a Hankel
operator with anti-analytic symbol $b$ mapping analytic functions into the
space of anti-analytic functions by $f \mapsto P_{-} b f$ is bounded with respect to the $L^2$ norm if and only if the symbol belongs to BMO. This theorem has an equivalent
formulation in terms of the boundedness of the commutator of the multiplication operator with symbol function $b$ and the Hilbert transform $[ H,b ] =H b-b H.$ 

Ferguson-Sadosky in \cite{FS} and later Ferguson-Lacey in their groundbreaking paper \cite{FL} study the symbols of bounded `big' and `little' Hankel operators on the bidisk through commutators of the tensor product or of the iterated form 
\[ [ H_{1} H_{2} ,b ] , \text{ and } [ H_{1} , [ H_{2} ,b ] ]  .\] Here $b=b ( x_{1} ,x_{2} )$ and the $H_{k}$ are the Hilbert transforms acting in
the $k^{\text{th}}$ variable. A full characterization of different two-parameter BMO spaces, Cotlar-Sadosky's little BMO and Chang-Fefferman's product BMO space, is given through these commutators. 

Through the use of completely different real variable methods, in {\cite{CRW}} Coifman-
Rochberg-Weiss   extended Nehari's one-parameter theory to real
analysis in the sense that the Hilbert transform was replaced by Riesz
transforms. These one-parameter results in {\cite{CRW}}  were treated in the multi-parameter setting in Lacey-Petermichl-Pipher-Wick {\cite{LPPW}}. Both the upper and lower estimate have proofs very different from those in one parameter.  
In addition, in both cases it is observed that the Riesz transforms are a representative testing class in the sense that BMO also ensures boundedness for (iterated) commutators with more general Calderon-Zygmund operators, a result now known in full generality due to Dalenc-Ou \cite{DO}. Notably the Riesz commutator has found striking applications to compensated compactness and div-curl lemmas, \cite{CLMS}, \cite{LPPW3}.

Our extension to the multi-parameter setting is two-fold. On the one hand we replace the Calderon-Zygmund operators by Journ\'e operators $J_i$ and on the other hand we also iterate the commutator: $$[J_1,...,[J_t,b]...].$$ We prove the remarkable fact that a multi-parameter BMO class still ensures boundedness in this situation and that the collection of tensor products of Riesz transforms remains the representative testing class. The BMO class encountered is a mix of little BMO and product BMO that we call a little product BMO. Its precise form depends upon the distribution of variables in the commutator. Our result is new even when no iterations are present: in this case, lower estimates were only known in the case of the double Hilbert transform \cite{FS}. The sufficiency of the little BMO class for boundedness of Journ\'e commutators had never been observed. 

\medskip

It is a general fact that two-sided commutator estimates have an equivalent formulation in terms of weak factorization. We find the pre-duals of our little product BMO spaces and prove a corresponding weak factorization result.

\medskip

Necessity of the little product BMO condition is shown through a lower estimate on the commutator. There is a sharp contrast when tensor products of Riesz transforms are considered instead of multiple Hilbert transforms and  when iterations are present. 

\medskip

In the Hilbert transform case, Toeplitz operators with operator symbol arise naturally.
Using Riesz transforms in $\mathbb{R}^d$ as a replacement, there is an absence of analytic structure and tools relying on analytic projection or orthogonal spaces are not readily available. We overcome part of this difficulty through the use of Calder\'on-Zygmund operators whose Fourier multiplier symbols are adapted to cones. This idea is inspired by {\cite{LPPW}}. Such operators are also mentioned in \cite{U}.
A class of operators of this type classifies little product BMO through two-sided commutator estimates, but it does not allow the passage to a classification through iterated commutators with tensor products of Riesz transforms. In a second step, we find it necessary to consider upper and lower commutator estimates using a well-chosen family of Journ\'e operators that are not of tensor product type. Through geometric considerations and an averaging procedure of zonal harmonics on products of spheres, we construct the multiplier of a special Journ\'e operator that preserves lower commutator estimates and resembles the multiple Hilbert transform: it has large plateaus of constant values and is a polynomial in multiple Riesz transforms. We expect that this construction allows other applications.

 There is an increase in difficulty when the dimension is greater than two, due to the simpler structure of the rotation group on $\mathbb{S}^1$. In higher dimension, there is a rise in difficulty when tensor products involve more than two Riesz transforms.
 
 \medskip

The actual passage to the Riesz transforms requires a stability estimate in commutator norms for certain multi-parameter singular integrals in terms of the mixed BMO class. In this context, we prove a qualitative upper estimate for iterated commutators using paraproduct free Journ\'e operators. We make use of recent versions of $T(1)$ theorems in this setting. These recent advances are different from the corresponding theorem of Journ\'e \cite{J2}. The results we allude to have the additional feature of providing a convenient representation formula for bi-parameter in {\cite{M}} and even multi-parameter in {\cite{O2}} Calder\'on-Zygmund operators by dyadic shifts.

\section{Aspects of Multi-Parameter Theory}

This section contains some review on Hardy spaces in several parameters as well as some new definitions and lemmas relevant to us. 

\subsection{Chang-Fefferman BMO}

We describe the elements of product Hardy space theory, as developed by
Chang and Fefferman as well as Journ{\'e}. By this we mean the Hardy
spaces associated with domains like the poly-disk or $\mathbb{R}^{\vec{d}} :=\bigotimes_{s=1}^{t} \mathbb{R}^{d_{s}}$ for $\vec{d}=(d_1,\ldots , d_t)$.
While doing so, we typically do not distinguish whether we are working on $\mathbb{R}^{\vec{d}}$ or $\mathbb{T}^d$. In higher dimensions, the Hilbert transform is usually replaced by the collection of Riesz transforms. 

The (real) one-parameter Hardy space $H^{1}_{\text{Re}} ( \mathbb{R}^{d} )$  denotes the class
of functions with the norm
\[ \sum_{j=0}^{d} \|{R_{j} f}\|_{1} \]
where $R_{j}$ denotes the $j^{\text{th}}$ Riesz transform or the Hilbert transform if the dimension is one. Here and below we adopt the
convention that $R_{0}$, the $0^{\text{th}}$ Riesz transform, is the identity. This
space is invariant under the one-parameter family of isotropic dilations,
while the product Hardy space $H^{1}_{\text{Re}} ( \mathbb{R}^{\vec{d}} )$ is invariant under dilations of each
coordinate separately. That is, it is invariant under a $t$ parameter family
of dilations, hence the terminology `multi-parameter' theory.
One way to define a norm on $H^{1}_{\text{Re}} ( \mathbb{R}^{\vec{d}} )$ is \[
\|f\|_{H^1}\sim\sum_{0\leq j_l\leq d_l}\|\bigotimes_{l=1}^{t}\mathrm{R}_{l,j_l}f\|_1.
\]
$\mathrm{R}_{l,j_l}$ is the Riesz transform in the $j_l^{\text{th}}$ direction of the $l^{\text{th}}$ variable, and the $0^{\text{th}}$ Riesz transform is the identity operator.

The dual of the real Hardy space $H^{1}_{\text{Re}} ( \mathbb{R}^{\vec{d}} )^{\ast} $ is
$\text{BMO} ( \mathbb{R}^{\vec{d}} )$, the $t$-fold product BMO space. It is a
theorem of S.-Y.~Chang and R.~Fefferman {\cite{CF1}}, {\cite{CF2}} that this space has a
characterization in terms of a product Carleson measure.

Define
\begin{equation}
  \label{e.BMOdef} 
  \lVert b \rVert_{\text{BMO} ( \mathbb{R}^{\vec{d}} )}
  := \sup_{U \subset \mathbb{R}^{\vec{d}}} 
  \left(\vphantom{\sum}\right. |{U}|^{-1} 
  \sum_{R \subset U} 
  \sum_{\vec{\varepsilon} \in { \text{sig}_{\vec{d}}}}
  |{\langle b,w_{R}^{\vec{\varepsilon}} \rangle}|^{2} \left.\vphantom{\sum}\right)^{1/2} .
\end{equation}
Here the supremum is taken over all open subsets $U \subset
\mathbb{R}^{\vec{d}}$ with finite measure, and we use a wavelet basis
$w_{R}^{\vec{\varepsilon}}$ adapted to rectangles $R=Q_{1} \times \cdots \times Q_{t}$, where
each $Q_{l}$ is a cube. The superscript $\vec{\varepsilon}$ reflects the fact that multiple wavelets are associated to any dyadic cube, see \cite{LPPW} for details. 
The fact that the supremum admits all open sets of finite measure cannot be omitted, as Carleson's example shows \cite{C}. This fact is responsible for some of the difficulties encountered when working with this space.
\begin{theorem}
  [Chang, Fefferman]\label{t.changfefferman} We have
  the equivalence of norms
  \[ \|{b}\|_{(H^{1}_{\text{Re}} ( \mathbb{R}^{\vec{d}} ))^{\ast}} \sim
     \|{b}\|_{\text{BMO} ( \mathbb{R}^{\vec{d}} )} .\]
  That is, ${ \text{BMO}} ( \mathbb{R}^{\vec{d}} )$ is the dual to
  $H^{1}_{\text{Re}} ( \mathbb{R}^{\vec{d}} )$.
\end{theorem}
This $\text{BMO}$ norm is invariant under a $t$-parameter family of dilations. Here
the dilations are isotropic in each parameter separately. See also \cite{F1} and \cite{F3}.

\subsection{Little BMO}
Following {\cite{CS}} and {\cite{FS}}, we recall some facts about the space little BMO, often written as `bmo',  and its predual.  A locally integrable function $b: \mathbb{R}^{\vec{d}}=\mathbb{R}^{d_1}\times \ldots \times \mathbb{R}^{d_s} \to \mathbb{C}$ is in
$\text{bmo}$ if and only if
\[ \| b \|_{\text{bmo}} = \sup_{\vec{Q} =Q_{1} \times \cdots \times Q_{s}} |
   \vec{Q} |^{-1} \int_{\vec{Q}} | b ( \vec{x} ) -b_{\vec{Q}} | < \infty \]
Here the $Q_{k}$ are $d_k$-dimensional cubes and $b_{\vec{Q}}$ denotes the
average of $b$ over $\vec{Q}$.

It is easy to see that this space consists of all
functions  that are uniformly in BMO in each variable
separately. Let $\vec{x}_{\hat{v}} = ( x_{1} , \ldots
.,x_{v-1} , \cdot , x_{v+1},\ldots ,x_{s} )$. Then $b(\vec{x}_{\hat{v}})$ is a function in $x_v$ only with the other variables fixed. Its BMO norm in $x_v$ is
\[ \| b ( \vec{x}_{\hat{v}} ) \|_{\text{BMO}} = \sup_{Q_{v}} | Q_{v}
   |^{-1} \int_{Q_{v}} | b ( \vec{x} ) -b ( \vec{x}_{\hat{v}}
   )_{Q_{v}} |dx_v \]
and the little BMO norm becomes 
$$\| b \|_{\text{bmo}} = \max_{v} \{
\sup_{\vec{x}_{\hat{v}}} \| b ( \vec{x}_{\hat{v}} )
\|_{\text{BMO}} \} .$$ On the bi-disk, this becomes $$\|b\|_{\text{bmo}}=\max \{\sup_{x_1}\|b(x_1,\cdot)\|_{\text{BMO}}, \sup_{x_2}\|b(\cdot,x_2)\|_{\text{BMO}}\},$$ the space discussed in {\cite{FS}}.
Here, the pre-dual is the space $H^1(\mathbb{T})\otimes L^1(\mathbb{T})+ L^1(\mathbb{T})\otimes H^1(\mathbb{T})$. All other cases are an obvious generalization, at the cost of notational inconvenience.

\subsection{Little product BMO}

In this section we define a BMO space which is in between little BMO and product BMO. As mentioned in the introduction, we aim at characterizing BMO spaces consisting for example of those functions $b(x_1,x_2,x_3)$ such that $b(x_1,\cdot,\cdot)$ and $b(\cdot,\cdot,x_3)$ are uniformly in product BMO in the remaining two variables. 
\begin{definition}\label{definitionlpbmo}
Let $b : \mathbb{R}^{\vec{d}} \to \mathbb{C}$ with $\vec{d}=(d_1,\cdots,d_t)$. Take a partition $\mathcal{I}=\{I_s:1\le s\le l\}$ of $\{1,2,...,t\}$ so that $\dot{\cup}_{1\le s \le l} I_s=\{1,2,...,t\}$.  We say that $b \in \text{BMO}_{\mathcal{I}}(\mathbb{R}^{\vec{d}})$ if for any choices ${\boldsymbol{v}}=(v_s), v_s\in I_{s}$, $b$ is uniformly in product BMO in the variables indexed by ${v_s}$.
We call a $\text{BMO}$  space of this type a `little product BMO'. If for any $\vec{x}=(x_1,...,x_t) \in \mathbb{R}^{\vec{d}}$, we define $\vec{x}_{\hat{\boldsymbol{v}}}$ by removing those variables indexed by ${v_s}$, the little product BMO norm becomes $$\|b\|_{\text{BMO}_{\mathcal{I}}}=\max_{{\boldsymbol{v}}} \{\sup_{\vec{x}_{\hat{\boldsymbol{v}}}}\|b(\vec{x}_{\hat{\boldsymbol{v}}})\|_{\text{BMO}}\}$$ where the BMO norm is product BMO in the variables indexed by ${v_s}$. 
\end{definition}
For example, when $\vec{d}=(1,1,1)=\vec{1}$,  when $t=3$ and $l=2$ with $I_1=(13)$ and $I_2=(2)$, writing $\mathcal{I}=(13)(2)$ the space  $\text{BMO}_{(13)(2)}(\mathbb{T}^{\vec{1}})$ arises,
which consists  of those functions that are uniformly in product BMO in the 
variables $(1,2)$ and $(3,2)$ respectively, as described above.  
Moreover, as degenerate cases, it is easy 
to see that $\text{BMO}_{(12\ldots t)}$ and $\text{BMO}_{(1)(2)\ldots (t)}$ are 
exactly little BMO and product BMO respectively,  the spaces we are familiar 
with.

Little product BMO spaces on $\mathbb{T}^{\vec{d}}$ can be defined in the same way. 
Now we find the predual of $\text{BMO}_{(13)(2)},$ which is a good model for other cases. We choose the order of variables most convenient for us.
\begin{theorem}\label{predual}
The pre-dual of the space $\text{BMO}_{(13)(2)}(\mathbb{T}^{\vec{1}})$ is equal to the 
space 
\[
\begin{split}
&H^1_{\text{Re}}(\mathbb{T}^{(1,1)})\otimes L^1(\mathbb{T})+L^1(\mathbb{T})\otimes 
H^1_{\text{Re}}(\mathbb{T}^{(1,1)})\\
&:=\{f+g: f\in H^{1}_{\text{Re}} ( {\mathbb{T}}^{(1,1)}) \otimes
 L^{1} ( \mathbb{T} ) \text{ and }  g \in L^{1} ( \mathbb{T} ) \otimes 
H^{1}_{\text{Re}} ( {\mathbb{T}}^{(1,1)})\}.
\end{split}
\]
\end{theorem}

\begin{proof}
  The space 
  $$H^{1}_{\text{Re}} ({ \mathbb{T}}^{(1,1)} ) \otimes L^{1} (\mathbb{T} ) = 
\{ f \in L^{1} ( \mathbb{T}^{3} ) :H_{1} f,H_{2} f, H_{1}H_{2} f \in 
L^{1} ( \mathbb{T}^{3} ) \}$$ 
  equipped with the norm 
  $\| f \| = \| f \|_{1} + \| H_{1} f \|_{1} + \| H_{2} f \|_{1} + \| 
H_{1} H_{2} f \|_{1}$
   is a Banach space. Let $W^1 = L^1(\mathbb{T}^3) \times 
L^1(\mathbb{T}^3) \times L^1(\mathbb{T}^3) \times L^1(\mathbb{T}^3) $  
equipped with the norm 
  $$\| (f_1, f_2 , f_3 , f_4) \|_{W_1}  = \| f _1\|_1 +   \| f _2\|_1 +   
\| f _3\|_1 +   \| f _4\|_1. $$
  Then we see that  $H^{1}_{\text{Re}} ( {\mathbb{T}}^{(1,1)}) \otimes L^{1} 
(\mathbb{T} )$ is isomorphically isometric to the closed subspace
  $$V = \{ (f, H_1(f) , H_2(f) , H_1 H_2(f)) : f\in H^1({\mathbb{T}}^{(1,1)}) 
\otimes L^1 ({\mathbb{T}})\} $$ 
  of $W^1.$ Now, the dual of $W^1$ is equal
  to $W^\infty = L^\infty(\mathbb{T}^3) \times L^\infty(\mathbb{T}^3) \times 
L^\infty(\mathbb{T}^3) \times L^\infty(\mathbb{T}^3)$ equipped with the norm
  $\|(g_1 , g_2 , g_3, g_4 )\|_\infty = \max \{ \|g_i \|_\infty 
: 1\le i \le 4 \}$ so the dual space of $V$ is equal to the quotient of 
$W^\infty$ by the annihilator $U$ of the subspace $V$ in $W^\infty.$ But, using the fact 
that the Hilbert transforms are self-adjoint up to a sign change, we see that  
  $$U = \{ (g_1, g_2 , g_3 , g_4) : g_1 + H_1 g_2 + H_2 g_3 + H_1 H_2 g_4 
= 0 \}$$ 
 and so:
$$V^\ast \cong W^\infty / U \cong \text{Im}(\theta )$$
where 
$$\theta (g_1 , g_2 ,g_3 , g_4) =g_1 + H_1 g_2 + H_2 g_3 + H_1 H_2 g_4 $$ since 
$U = \text{ker}(\theta ).$ But 
$$\text{Im}(\theta ) = L^\infty(\mathbb{T}^3) + H_1( L^\infty(\mathbb{T}^3))+
H_2(L^\infty(\mathbb{T}^3)) + H_1(H_2( L^\infty(\mathbb{T}^3)))$$ 
is equal to
the functions that are uniformly in product
  BMO in variables 1 and 2.

  Using the same reasoning we see that the dual of $L^1(\mathbb{T}) \otimes 
H^{1}_{\text{Re}} ({ \mathbb{T}}^{(1,1)} )$ is equal to $L^\infty(\mathbb{T}^3) + 
H_2(L^\infty(\mathbb{T}^3) ) + H_3 (L^\infty(\mathbb{T}^3) ) + 
H_2 H_3 (L^\infty(\mathbb{T}^3) )$,
  which is equal to the space of functions that are uniformly in product
  BMO in variables 2 and 3.  
  
  Now, we consider the `$L^1$ sum' of the spaces 
$ H^{1}_{\text{Re}} ({ \mathbb{T}}^{(1,1)} ) \otimes L^1(\mathbb{T})$ and 
$L^1(\mathbb{T}) \otimes H^{1}_{\text{Re}} ({ \mathbb{T}}^{(1,1)} )$; that is 
$$M_{(13)(2)} = \{(f,g): f \in  H^{1}_{\text{Re}} ({ \mathbb{T}}^{(1,1)} ) \otimes L^1(\mathbb{T});
g \in L^1(\mathbb{T}) \otimes H^{1}_{\text{Re}} ({ \mathbb{T}}^{(1,1)} )\}$$
equipped with the norm
$$\| (f,g) \| = \|f \|_{H^{1}_{\text{Re}} ({ \mathbb{T}}^{(1,1)} ) \otimes L^1(\mathbb{T})} + \| g \|_{L^1(\mathbb{T}) \otimes 
H^{1}_{\text{Re}} ({ \mathbb{T}}^{(1,1)} )}.
$$
We see that, if $\phi : M_{(13)(2)} \to L^1( ( \mathbb{T}^{3} )$ is defined
by $\phi (f,g) = f + g,$ then the image of $\phi$ is isometrically 
isomorphic to the quotient of $M_{(13)(2)}$ by the space 
\[
\begin{split}
N &= \{ (f,g) \in  M_{(13)(2)} : f + g = 0 \} \\
&= \{ (f,-f) : f \in 
H^{1}_{\text{Re}} ({ \mathbb{T}}^{(1,1)} ) \otimes L^1(\mathbb{T}) 
\cap L^1(\mathbb{T}) \otimes H^{1}_{\text{Re}} ({ \mathbb{T}}^{(1,1)} )\}.
\end{split}
\]
Now, recall that the dual of the quotient $M/N$ is equal to the annihilator of $N.$ 
It is easy to see that the annihilator of $N$ is equal to the set of ordered pairs 
$(\phi ,\phi)$ with 
$\phi$ in the intersection of the duals of the two spaces. Thus 
the dual of the image of $\theta$ is equal to  $\text{BMO}_{( 13 ) (2)}.$ The norm
of an element in the predual is equal to its norm as an element of the 
double dual which is easily computed.
\end{proof}

Following this example, the reader may easily find the correct formulation for the predual of other little product BMO spaces as well those in several variables, replacing the Hilbert transform by all choices of Riesz transforms. For instance, one can prove that the predual of the space $\text{BMO}_{(13)(2)}(\mathbb{R}^{\vec{d}})$ is equal to $H^1_{\text{Re}}(\mathbb{R}^{(d_1,d_2)})\otimes L^1(\mathbb{R}^{d_3})+L^1(\mathbb{R}^{d_1})\otimes H^1_{\text{Re}}(\mathbb{R}^{(d_2,d_3)})$.

\section{The Hilbert transform case}

In this section, we characterize the boundedness of commutators of the form $ [ H_{2} , [ H_{3} H_{1} ,b ] ]$ as  operators on 
$L^{2} ( \mathbb{T}^{3} )$. In the case of the Hilbert transform, this case is representative of the general case and provides a starting point that is easier to read because of the simplicity of the expression of products and sums of projection onto orthogonal subspaces. Its general form can be found at the beginning of Section \ref{section_riesz}.

Now let $b\in L^1 ({\mathbb{T}}^n)$ and let $P$ and $Q$ denote orthogonal projections onto subspaces of $L^2(\mathbb{T}^n)$. We shall describe relationships between functions in the little product BMOs and several types of projection-multiplication operators. These will be Hilbert transform-type operators of the form $P - P^{\perp}$;  and iterated Hankel or Toeplitz type operators of the form 
$Q^\perp b Q$ (Hankel), $P b P$ (Toeplitz),$ P Q^\perp b Q P$ (mixed), where $b$ means the (not a priori bounded) multiplication operator $M_b$ on $L^2(\mathbb{T}^n)$.  

We shall use the following simple observation concerning Hilbert transform type operators again and again:

\begin{remark}\label{Hilbobs}
 If  $H = P - P^\perp$  and $T:L^2(\mathbb{T}^n) \to L^2(\mathbb{T}^n) $ is a linear operator  then 
 $$ [H,T] = 2PTP^\perp - 2P^\perp T P$$
  and $H$ is bounded if and only if $PTP^\perp$ and $P^\perp T P$ are.
\end{remark}
 
\begin{proof}
 \begin{align*} 
 (P - P^\perp)T - T(P - P^\perp) &= (P - P^\perp)T(P + P^\perp) - (P + P^\perp)T(P - P^\perp)\\
 & = 2PTP^\perp - 2P^\perp T P.
 \end{align*}
\end{proof} 

We state the main result of this section. 
\begin{theorem}\label{hilbthm}
  Let $b \in L^{1} ( \mathbb{T}^{3} )$. Then the following are equivalent with linear dependence on the respective norms
  \begin{enumerate}[(1)]
    \item $b \in \text{BMO}_{( 13 ) (2)}$
    
    \item The commutators $[ H_{2} , [ H_{1} ,b ] ]$ and $[ H_{2} , [ H_{3} ,b
    ] ]$ are bounded on $L^{2} ( \mathbb{T}^{3} )$
    
    \item The commutator $ [ H_{2} , [ H_{3} H_{1} ,b ] ]$ is bounded on $L^{2} ( \mathbb{T}^{3} )$.
  \end{enumerate}
\end{theorem}

\begin{corollary}\label{hilbest}
We have the following two-sided estimate 
\[ \|b\|_{\text{BMO}_{(13)(2)}}\lesssim\| [ H_{2} , [ H_{3} H_{1} ,b ] ]\|_{L^{2} ( \mathbb{T}^{3} )\to L^{2} ( \mathbb{T}^{3} )}\lesssim\|b\|_{\text{BMO}_{(13)(2)}}.\]
\end{corollary}

 It will be useful to denote by $Q_{13}$ orthogonal projection on the subspace of functions which are either analytic or anti-analytic in the first and third variables;
$Q_{13} = P_1P_3  + P_1^\perp P_3^\perp$. Then the projection $Q_{13}^\perp $ onto the orthogonal of this subspace is defined by 
$Q_{13}^\perp = P_1^\perp P_3 + P_1 P_3^\perp.$  We reformulate properties 
{\it (2)} and {\it (3)} in the statement of Theorem \ref{hilbthm} in terms of Hankel Toeplitz type operators.

\begin{lemma}
  \label{hilblem}
  We have the following algebraic facts on commutators and projection operators.
  \begin{enumerate}[(1)]
   \item The commutators $[ H_{2} , [ H_{1} ,b ] ]$ and $[ H_{2} , [ H_{3} ,b
    ] ]$ are bounded on $L^{2} ( \mathbb{T}^{3} )$ if and only if  the operators
     $P_{i} P_{2} b P^{\perp}_{i} P^{\perp}_{2} ,
     P^{\perp}_{i} P_{2} b P_{i} P^{\perp}_{2} ,
     P_{i} P^{\perp}_{2} b P^{\perp}_{i} P_{2},
     P^{\perp}_{i} P^{\perp}_{2} b P_{i} P_{2}$
    with $i \in \{ 1,3 \}$ are bounded on $L^{2} ( \mathbb{T}^{3} )$.
    
   \item The commutator $ [ H_{2} , [ H_{3} H_{1} ,b ] ]
    $ is bounded on $L^{2} ( \mathbb{T}^{3} )$ if and only if all
    four operators
    $P_{2} Q_{13} b Q^{\perp}_{13} P^{\perp}_{2},
    P^{\perp}_{2} Q^{\perp}_{13} b Q_{13} P_{2} ,
    P_{2} Q^{\perp}_{13}b Q_{13} P^{\perp}_{2} ,
    P^{\perp}_{2} Q_{13} b Q^{\perp}_{13}P_{2}$  are bounded
    on $L^{2} ( \mathbb{T}^{3} )$. 

   \end{enumerate}
\end{lemma}

\begin{proof}
Using Remark \ref{Hilbobs} it is easy to see that 
\begin{equation*}
[ H_{2} , [ H_{1} ,b ] ] = 4\big{(}(P_2P_1bP_1^\perp P_2^\perp - P_2P_1^\perp bP_1P_2^\perp) - (P_2^\perp P_1bP_1^\perp P_2 - P_2^\perp P_1^\perp bP_1 P_2)\big{)}
\end{equation*} 
and that the corresponding equation for $[ H_{2} , [ H_{3} ,b]]$ is also true. This, along with the observation that  the ranges of
  all arising summands are mutually orthogonal, gives assertion {\it (1)}.
 To prove {\it (2)} we just notice that $H_1H_3 = Q_{13}  - Q_{13}^\perp$ is a Hilbert transform type operator which permits us to repeat the above argument replacing $P_1$ by $Q_{13}$.
 \end{proof}

The following lemma will allow us to insert an additional Hilbert transform into the commutator without reducing the norm.
  
  \begin{lemma}
 \label{ToeplitzLemma2}
 $\|P_3 P_1^\perp P_2^{\perp}  b P_1 P_2 P_3 \|_{L^2 \to L^2} = \|P_1^\perp P_2^{\perp}  b P_1 P_2 \|_{L^2 \to L^2}.$
 \end{lemma}
 
 \begin{proof}
 
  The inequality $\le$ is trivial, since $P_3$ is a projection which commutes with $ P_1^\perp$ and $ P_2^{\perp}.$ To see $\ge$, notice that
 $P_3 P_1^\perp P_2^{\perp}  b P_1 P_2 P_3$ is a Toeplitz operator with symbol $P_1^\perp P_2^{\perp}  b P_1 P_2$. So $\|P_3 P_1^\perp P_2^{\perp}  b P_1 P_2 P_3 \| = sup_{x_3}\|P_1^\perp P_2^{\perp}  b(\cdot,\cdot,x_3) P_1 P_2 \|.$ The latter is just $\|P_1^\perp P_2^{\perp}  b P_1 P_2 \|$. For convenience we include a sketch of the facts about Toeplitz operators we use.
 Let $W_3$ be the operator of multiplication by $z_3$,
$W_3(f) = z_3  f$, acting on  $L^2({\mathbb{T}}^3)$. If we define $B = P_1^\perp P_2^{\perp}  b P_1 P_2 $ as well as
$$A_n = W_3^{\ast n}(P_3 P_1^\perp P_2^{\perp}  b P_1 P_2 P_3) W_3^n \text{ 
and } C_n = W_3^{ n}(P^\perp_3 P_1^\perp P_2^{\perp}  b P_1 P_2 P^\perp_3) W_3^{\ast n}$$ as operators acting on $L^2({\mathbb{T}}^3)$
then the sequences $A_n$ and $C_n$ converge to $B$ in the strong operator topology: 
 it is easy to see that $W_3$ , $W_3^\ast$; and $P_3$ commute with $P_1, P_2 , P_1^\perp $ and $P_2^\perp$. The multiplier $b$ satisfies the equation
$ W_3^{\ast n} b W_3^n = b$ and $W_3^nW_3^{\ast n} = Id.$  So we see that
\[
A_n =P_1^\perp P_2^{\perp}(W_3^{\ast n}P_3 W_3^{ n}) b P_1 P_2(W_3^{\ast n} P_3 W_3^n).
\]
But if $f \in  L^2({\mathbb{T}}^3)$, then, since $W_3^n$ is a unitary operator:
$$\| W_3^{\ast n} P_3 W_3^n(f) - f\| = \|P_3 W_3^n(f) - W_3^n(f)\| = \| (P_3 - I)(W_3^n)(f)\| \to 0\ \ \ \ (n\to \infty),$$
as tail of a convergent Fourier series.
This means that $W_3^{\ast n} P_3 W_3^n $
converges to the identity in the strong operator topology. Thus, for each $f\in  L^2({\mathbb{T}}^3)$ we have $\| (A_n - B) (f) \| \to 0$.
So
\begin{align*}
 \| P_1^\perp P_2^{\perp}  b P_1 P_2  \| 
&\le\sup_{n\in {\mathbb{N}}}\| W_3^{\ast n}(P_3 P_1^\perp P_2^{\perp}  b P_1 P_2 P_3) W_3^n\| \cr
&\le\| P_3 P_1^\perp P_2^{\perp}  b P_1 P_2 P_3\|,\cr
\end{align*}
\end{proof}
 
Now, we are ready to proceed with the proof of the main theorem of this section.
  
 \begin{proof}
  (of Theorem \ref{hilbthm}) We show ${\it (1)} \Leftrightarrow {\it (2)}$ and ${\it (2)} \Leftrightarrow {\it (3)}$.
  
  ${\it (1)} \Leftrightarrow {\it (2)}$.
  Consider $f=f ( x_{1} ,x_{2} )$ and $g=g ( x_{3} )$. Then $[ H_{2}
  , [ H_{1} ,b ] ] ( f g ) =g \cdot [ H_{2} , [ H_{1} ,b ] ] ( f ) .$ So $\| [
  H_{2} , [ H_{1} ,b ] ] ( f g ) \|^{2}_{L^{2} ( \mathbb{T}^{3} )} = \| F g
  \|^{2}_{L^{2} ( \mathbb{T} )}$ where $F ( x_{2} ) = \| [ H_{2} , [ H_{1} ,b
  ] ] ( f ) \|_{L^{2} ( \mathbb{T}^{2} )}$. The map $g \mapsto F g$ has
  $L^{2} ( \mathbb{T} )$ operator norm $\| F \|_{\infty}$. Now change the
  roles of $x_{1}$ and $x_{3}$. The Ferguson-Lacey equivalences $\| [ H_{2} ,
  [ H_{i} ,b ] ] \| \sim \| b \|_{\text{BMO}}$ give the desired result.

${\it (2)} \Rightarrow {\it (3)}$.
Boundedness of the commutators $[ H_{2} , [ H_{1} ,b ] ]$ and $ [ H_{2} , [H_{3} ,b ] ]$ implies the boundedness of  the mixed commutator $[ H_{2} , [ H_{1} H_{3},b]]$ by the identity $[ H_{2} , [ H_{1} H_{3},b]]=H_1[H_2,[H_3,b]]+[H_2,[H_1,b]]H_3$.
  
  ${\it (3)} \Rightarrow {\it (2)}$.
  This part relies on Lemma \ref{ToeplitzLemma2}.
  We wish to conclude from the boundedness of $ [ H_{2} , [ H_{3}
  H_{1} ,b ] ] $ the boundedness of $[ H_{2} , [ H_{1} ,b ] ]$ and
  $[ H_{2} , [ H_{3} ,b ] ]$. To see boundedness of $[ H_{2} , [ H_{1} ,b ]
  ]$, let us look at one of the Hankels from Lemma \ref{hilblem}.   Lemma \ref{ToeplitzLemma2} shows that $P^{\perp}_{2} P^{\perp}_{1}{bP}_{2} P_{1} $ is bounded if and only if  the operator $P_3 P_1^\perp P_2^{\perp}  b P_1 P_2 P_3$ is. And the latter is an operator found in the list from part {\it (2)} of Lemma \ref{hilblem}. The analogous reasoning shows that all eight Hankels in \ref{hilblem} are bounded and so {\it (2)} is proved.
  \end{proof}

\section{Real variables: lower bounds}\label{section_riesz}

In this section, we are again in $\mathbb{R}^{\vec{d}}$ with $\vec{d}=(d_1,\ldots ,d_t)$ and a partition $\mathcal{I}=(I_s)_{1\le s \le l}$ of $\{1,\ldots ,t\}$. It is our aim to prove the following characterization theorem of the space $\text{BMO}_{\mathcal{I}}(\mathbb{R}^{\vec{d}})$.

\begin{theorem}\label{theorem_riesz} The following are equivalent with linear dependence of the respective norms.
\begin{enumerate}[(1)]
\item $b \in \text{BMO}_{\mathcal{I}}(\mathbb{R}^{\vec{d}})$
\item All commutators of the form $[R_{k_1,j_{k_1}},\ldots ,[R_{k_l,j_{k_l}},b]\ldots ]$ are bounded in $L^2(\mathbb{R}^{\vec{d}})$ where $k_s\in I_s$  and $R_{k_s,j_{k_s}}$ is the one-parameter Riesz transform in direction $j_{k_s}$.
\item All commutators of the form $[\vec{R}_{1,\vec{j}^{(1)}},\ldots,[\vec{R}_{l,\vec{j}^{(l)}},b]\ldots] $ are bounded in $L^2(\mathbb{R}^{\vec{d}})$ where $\vec{j}^{(s)}=(j_k)_{k\in I_s}$, $1\le j_k\le d_k$ and the operators $\vec{R}_{s,\vec{j}^{(s)}}$ are a tensor product of Riesz transforms $\vec{R}_{s,\vec{j}^{(s)}}=\bigotimes_{k\in I_s}R_{k,j_k}$. 

\end{enumerate}
\end{theorem}

Such two-sided estimates also hold in $L^p$ for $1<p<\infty$. Remarks will be made in section \ref{generalcase}. From the inductive nature of our arguments, it will also be apparent that the characterization holds when we consider intermediate cases, meaning commutators with any fixed number of Riesz transforms in each iterate. Below we state our most general two-sided estimate through Riesz transforms.
\begin{theorem}\label{theorem_riesz_choice_number}
Let $1<p<\infty$. Under the same assumptions as Corollary \ref{corollary_riesz} and for any fixed $\vec{n}=(n_s)$ where $1\le n_s \le |I_s|$, we have the two-sided estimate 
$$\|b\|_{\text{BMO}_{\mathcal{I}}(\mathbb{R}^{\vec{d}})}\lesssim\sup_{\vec{j}}\|[\vec{R}_{1,\vec{j}^{(1)}},\ldots,[\vec{R}_{l,\vec{j}^{(l)}},b]\ldots] \|_{L^p(\mathbb{R}^{\vec{d}})\righttoleftarrow}\lesssim\|b\|_{\text{BMO}_{\mathcal{I}}(\mathbb{R}^{\vec{d}})}$$  where $\vec{j}^{(s)}=(j_k)_{k\in I_s}$, $0\le j_k\le d_k$ and for each $s$, there are $n_s$ non-zero choices. A Riesz transform in direction 0 is understood as the identity.
\end{theorem}
For $p=2$ and $\vec{n}=\vec{1}$ this is the equivalence {\it{(1)}} $\Leftrightarrow $ {\it{(2)}} and for $\vec{n}=(|I_1|,\ldots,|I_l|)$ it is the equivalence {\it{(1)}} $\Leftrightarrow$ {\it{(3)}} from Theorem \ref{theorem_riesz}.

Our main focus is of course on a two-sided estimate when $\vec{n}=(|I_1|,\ldots,|I_l|)$ when the tensor product is a paraproduct-free Journ\'e operator:

\begin{corollary}\label{corollary_riesz}
Let $\vec{j}=(j_1,\ldots,j_t)$ with $1\le j_k\le d_k$ and let for each $1\le s\le l$, $\vec{j}^{(s)}=(j_k)_{k\in I_s}$  be associated a tensor product of Riesz transforms $\vec{R}_{s,\vec{j}^{(s)}}=\bigotimes_{k\in I_s}R_{k,j_k}$; here the $R_{k,j_{k}}$ are $j_k^{\text{th}}$ Riesz transforms acting on functions defined on the $k^{\text{th}}$ variable. 
We have the two-sided estimate $$\|b\|_{BMO_{\mathcal{I}}(\mathbb{R}^{\vec{d}})} \lesssim \sup_{\vec{j}}\|[\vec{R}_{1,\vec{j}^{(1)}},\ldots,[\vec{R}_{t,\vec{j}^{(t)}},b]\ldots]\|_{L^p(\mathbb{R}^{\vec{d}})\righttoleftarrow}\lesssim \|b\|_{BMO_{\mathcal{I}}(\mathbb{R}^{\vec{d}})}.$$
\end{corollary}

The statements above also serve as the statement of the general case for products of Hilbert transforms. In fact, when any $d_k=1$ just replace the Riesz transforms by the Hilbert transform in that variable. In this section, we consider the case $d_k\ge 2$ for $1\le k\le t$ and thus iterated commutators with tensor products of Riesz transforms only. The special case when $d_k=1$ for some $k$ is easier but requires extra care for notation, which is why we omit it here.

The proof in the Hilbert transform case relied heavily on analytic projections and orthogonal spaces, a feature that we do not have when working with Riesz transforms. We are going to simulate the one-dimensional case by a two-step passage via intermediary Calder\'on-Zygmund operators whose multiplier symbols are adapted to cones. 

In dimension $d \geqslant 2$, a cone $C \subset$ $\mathbb{R}^{d}$  with cubic base is given by
the data $( \xi ,Q )$ where $\xi \in \mathbb{S}^{d-1}$ is the direction
of the cone and the cube $Q \subset \xi^{\perp}$
centered at the origin is its aperture. The cone consists of all vectors
$\theta$ that take the form $( \theta_{\xi}   \xi  , \theta_{\bot} ) $ 
where $   \theta_{\xi} = \langle \theta ,\xi \rangle$ and $\theta_{\bot} \in
\theta_{\xi} Q.$ By $\lambda C$ we mean the dilated cone with data $( \xi
, \lambda Q )$.

A cone $D$ with ball base has data $(\xi,r)$ for $0<r<\pi/2$ and $\xi \in \mathbb{S}^{d-1}$ and consists of the vectors $\{\eta \in \mathbb{R}^{d}:d(\xi,\eta/\|\eta\|) \le r\}$ where $d$ is the geodesic distance (with distance of antipodal points being $\pi$.) 

 Given any cone $C$ or $D$, we consider its Fourier projection operator defined via
$\widehat{P_{C}} f=\chi_{C} \hat{f} .$ When the
apertures are cubes, such operators are combinations of Fourier projections
onto half spaces and as such admit uniform $L^{p}$ bounds. Among others, this fact made cubic cones necessary in the considerations in \cite{LPPW} and \cite{DP} that we are going to need. For further technical
reasons in the proof 
these operators are not quite good enough, mainly
because they are not of Calder\'on-Zygmund type. For a given cone $C$, consider a Calder\'on-Zygmund operator
$T_{C}$ with a kernel $K_{C}  $ whose Fourier symbol $\widehat{K_{C}} \in
C^{\infty}$ and satisfies the estimate $\chi_{C} \leqslant
\widehat{K_{C}} \leqslant \chi_{( 1+ \tau ) C}$. This is accomplished
by mollifying the symbol $\chi_{C}$ of the cone projection associated
to cone $C$ on $\mathbb{S}^{d -1}$ and then extending radially. We use the same definition for $T_{D}$.

Given a collection of cones $\vec{C} = ( C_{k}
)$ we denote by $T_{\vec{C}} ,P_{\vec{C}}$ the corresponding tensor
product operators.

In {\cite{LPPW}} it has been proved that Calder\'on-Zygmund operators adapted to
certain cones of cubic aperture classify product BMO via commutators.  As part of the argument, it was observed that test functions with opposing Fourier supports made the commutator large. 
In {\cite{DP}} a refinement was proven, that will be helpful to us. We prefer to work with cones with round base. Lower bounds for such commutators can be deduced from the assertion of the main theorem in \cite{DP}, 
but we need to preserve the information on the Fourier support of the test function in order to succeed with our argument. Information on this test function is instrumental to our argument: it reduces the terms arising in the commutator to those resembling Hankel operators.  We have the following lemma, very similar to that in \cite{LPPW} and \cite{DP}, the only difference being that the cones are based on balls instead of cubes.

\begin{lemma}\label{lemma_testfunction_ball} For every parameter $1\le k\le t$ there exist a finite set of directions $\Upsilon_k\in \mathbb{S}^{d_k-1}$ and an aperture $0<r_k<\pi/2$ so that,
for every symbol $b$ belonging to product BMO, there exist cones $D_{k}=D(\xi_k,r_k)$ with $\xi_k\in \Upsilon_k$  as well as a normalised test function $f=\bigotimes_{k=1}^t f_k$ whose components have Fourier support in the opposing  cones $D(-\xi_k,r_k)$ such that 
  \[ \|[T_{1,D_{1}} ...,[T_{t,D_{t}} ,b]...] f \|_{2} \gtrsim \| b
     \|_{\text{BMO}_{(1)...(t)}(\mathbb{R}^{\vec{d}})} . \]
\end{lemma}
The stress is on the fact that the collection is finite, somewhat specific and serves all admissible product BMO functions. 

\begin{proof}
The lemma in \cite{DP} supplies us with the sets of directions $\Upsilon_k$ as well as cones of cubic  aperture $Q_k$ and a test function $f$ supported in the opposing cones. Now choose the aperture $r_k$ large enough so that $(1+\tau)C(\xi_k,Q_k)\subset D(\xi_k,r_k)$.  Then we have the commutator estimate  \[ \|[T_{1,D_{1}} ...,[T_{t,D_{t}} ,b]...] f\|_{2} \gtrsim \| b\|_{\text{BMO}_{(1)...(t)}(\mathbb{R}^{\vec{d}})} . \] 
In fact, both commutators with cones $C$ and $D$ are $L^2$ bounded and reduce to $\|T_{\vec{D}}(b f) \|_{2}$ or $\|T_{\vec{C}}(b f) \|_{2}$ respectively thanks to the opposing Fourier support of $f$. Observe that $T_{\vec{C}}(b f)=T_{\vec{D}}(T_{\vec{C}}(b f))=T_{\vec{C}}(T_{\vec{D}}(b f))$.  With $\|T_{\vec{C}}\|_{2\to 2}\le 1$, we see that $\|T_{\vec{D}}(b f)\|_2\ge \|T_{\vec{C}}(b f)\|_2$. 
\end{proof}

Using this a priori lower estimate,  we are going to prove the lemma below. 

\begin{lemma}\label{lemma_cone}
Let us suppose we are in $\mathbb{R}^{\vec{d}}$ with $\vec{d}=(d_1,\ldots,t)$ and a partition $\mathcal{I}=(I_s)_{1\le s \le l}$. For every $1\le k \le t$ there exists a finite set of directions $\Upsilon_k \subset
\mathbb{S}^{d_k - 1}$ and an aperture $r_k$ 
so that the following hold for all $b \in \text{BMO}_{\mathcal{I}}(\mathbb{R}^{\vec{d}}):$
\begin{enumerate}[(1)]
  \item  For every $1\le s \le l$ there exists a coordinate $v_s\in I_s$ and a direction $\xi_{v_s} \in \Upsilon_{v_s}$ and so that with the choice
  of cone  $D_{v_s}=D ( \xi_{v_s}, r_{v_s})$ and arbitrary $D_{k}$ 
  for coordinates $k\in I_s \setminus \{v_s\}$ and if $\vec{D}_s$ denotes their tensor product, then we have  $$\| [ {T}_{1,\vec{D}_1}, \ldots ,[ {T}_{l,\vec{D}_l}, b]\ldots]
  \|_{2 \rightarrow 2} \gtrsim \| b \|_{\text{BMO}_{\mathcal{I}}(\mathbb{R}^{\vec{d}})},$$

  \item The test function $f=\bigotimes_{k=1}^t f_k$ which gives us a large $L^2$ norm in {\it (1)} has
  Fourier supports of the $f_k$ contained in $D(-\xi_k,r_k)$ when $k=v_s$ and in $D_k$ otherwise.
  \end{enumerate}
  
\end{lemma}

Before we can begin with the proof of Lemma \ref{lemma_cone}, we will need a real variable version of the facts on Toeplitz operators used earlier. 
\begin{lemma} \label{lemma_realtoeplitz}
Let $D_{k}$ for  $1\le k\le t$ denote any cones with respect to the $k^{\text{th}}$ variable. Let $T_{D_{k}}$ denote the adapted Calder\'on-Zygmund
  operators. Let $K$ be any proper subset of $\{k: 1\le k\le t\}$, let $\vec{D}_K=\bigotimes_{k\in K} D_k$ and $T_{\vec{D}_K}$ the associated tensor product of Calder\'on-Zygmund operators. Let  $P^{\sigma}_{\vec{D}_K}$ be a tensor product of projection operators on cones $D(\xi_k,r_k)$ or opposing cones $D(-\xi_k,r_k)$. Let $j\notin K$.
Then 
  $$ \|T_{\vec{D}_K}T_{D_{j}}bP^{\sigma}_{\vec{D}_K}P_{D_j}\|_{L^2(\mathbb{R}^{\vec{d}})\righttoleftarrow}=\|T_{\vec{D}_K}bP^{\sigma}_{\vec{D}_K}\|_{L^2(\mathbb{R}^{\vec{d}})\righttoleftarrow}.$$
\end{lemma}

\begin{proof}
   
We will establish this by composing some unilateral shift operators and studying their Fourier transform in the $j$ variable.
Let $\xi_j$ denote the direction of the cone $D_j$, for any $l$ define the shift operator
\[
S_l g(x_j)=\int_{\mathbb{R}^{d_j}} \hat{g}(\eta_j) e^{2\pi i(l\xi_j+\eta_j)x_j}\,d\eta_j.
\]
$S_l$ is a translation operator on the Fourier side along the direction $\xi_j$ of the cone $D_j$. It is not hard to observe that $S_l^{\ast}=S_{-l}$.
Now define
\[
A_l=S_{-l}T_{\vec{D}_K}T_{D_{j}}bP^{\sigma}_{\vec{D}_K}P_{D_{j}}S_l, \text{ and }
B=T_{\vec{D}_K}bP^{\sigma}_{\vec{D}_K}.
\]
We will prove that as $l\to +\infty$, $A_l\to B$ in the strong operator topology. As in the argument in Lemma \ref{ToeplitzLemma2}, this together with the fact that $S_l$ is an isometry will complete the proof.
To see the convergence, let's first remember that $S_l$ only acts on the $j$ variable, and one always has the identities 
\[
S_l S_{-l}=Id\,\quad\text{and} \quad S_{-l}bS_l=b.
\]
This implies
\[
\begin{split}
A_l&=T_{\vec{D}_{K}}(S_{-l}T_{D_{j}}S_{l})(S_{-l}bS_{l}){P^{\sigma}_{\vec{D}_K}}(S_{-l}P_{D_{j}}S_l)\\
&=T_{\vec{D}_{K}}(S_{-l}T_{D_{j}}S_{l})b{P^{\sigma}_{\vec{D}_K}}(S_{-l}P_{D_{j}}S_l).
\end{split}
\]
We claim that both $S_{-l}T_{D_{j}}S_{l}$ and $S_{-l}P_{D_j}S_l$ converge to the identity operator in the strong operator topology, which then implies that $A_l\to B$ as $l\to \infty$. 
We will only prove $S_{-l}T_{D_{j}}S_{l}\to Id$ as the second limit is almost identical. Observe that $\|S_{-l }T_{D_{j}}S_{l}f-f\|=\|(T_{D_{j}}-I)S_{l}f\|$.
Given any $L^2$ function $f$ and any fixed large $l\geq 0$. Consider the $f$ with frequencies supported in $\mathbb{R}^{d_1}\times \ldots \times (D_j-l\xi_j)\times \ldots \times \mathbb{R}^{d_t}$. In this case, $S_lf$ has Fourier support in $\mathbb{R}^{d_1}\times \ldots \times D_j \times \ldots \times \mathbb{R}^{d_t}$ where the symbol of $T_{D_{j}}$ equals 1. Thus, for such $f$, we have $S_{-l}T_{D_{j}}S_lf=f$. The sets $\mathbb{R}^{d_1}\times \ldots \times (D_j-l\xi_j)\times \ldots \times \mathbb{R}^{d_t}$ exhaust the frequency space. With $\|T_{D_{j}}-I\|_{2\to 2}\le 1$ the operators  $S_{-l}T_{D_{j}}S_{l}$ converge to the Identity in the strong operator topology, and the lemma is proved. Observe that the aperture of the cone $D_j$ is not relevant to the proof.
\end{proof}

We proceed with the proof of the lower estimate for cone transforms.

\begin{proof} (of Lemma \ref{lemma_cone})
For a given symbol $b\in \text{BMO}_{\mathcal{I}}$, there exist for all $1\le s\le l$ coordinates ${\boldsymbol{v}}=(v_s), v_s\in I_s$ and a choice of variables not indexed by $v_s$, $\vec{x}^{\;0}_{\hat{\boldsymbol{v}}}$ so that 
   up to an arbitrarily small error
  \[ \| b \|_{\text{BMO}_{\mathcal{I}}} = \| b ( \vec{x}^{\;0}_{\hat{\boldsymbol{v}}} )
     \|_{\text{BMO}_{(v_1)\ldots(v_l)} }. \]
  By Lemma \ref{lemma_testfunction_ball}, there exist cones $D_{v_s}=D(\xi_{v_s},r_{v_s})$ with directions $\xi_{v_s} \in \Upsilon_{v_s}$ and a normalised test function $f_H$ in variables $v_s$ with opposing Fourier support such that we have the lower estimate $$\| [ T_{v_1,D_{v_1}}, \ldots ,[ T_{v_l,D_{v_l}} ,b (  \vec{x}^{\;0}_{\hat{\boldsymbol{v}}} )
  ] \ldots] (f_H)\|_{L^{2} (\mathbb{R}^{\vec{d}_{\boldsymbol{v}}}) }
  \gtrsim 
  \|b ( \vec{x}^{\;0}_{\hat{\boldsymbol{v}}})\|_{\text{BMO}_{(v_1)\ldots(v_l)}}$$
  where $\mathbb{R}^{\vec{d}_{\boldsymbol{v}}}=  \mathbb{R}^{d_{v_1}} \times \ldots \times \mathbb{R}^{d_{v_l}}$.
  
    We now consider the commutator with the same cones but with full symbol $b=b(
  \cdot , \ldots , \cdot )$. Due to the lack of action on the variables not indexed by $v_s$,
  in the commutator, we have $$[T_{v_1,D_{v_1}},\ldots ,[T_{v_l,D_{v_l}} ,b]\ldots] (
  f_Hg ) =g \cdot [T_{v_1,D_{v_1}},\ldots ,[T_{v_l,D_{v_l}} ,b ]\ldots] ( f_H)$$ for $g$ that only depends upon variables not indexed by $v_s$. Again using that multiplication operators
  in $L^{2}$ have norms equal to the $L^{\infty}$ norm of their symbol, for
  the `worst' $L^{2}$-normalised $g$ we have
  \[
  \begin{split}
   & \|[T_{v_1,D_{v_1}},\ldots ,[T_{v_l,D_{v_l}} ,b ]\ldots] (
    f_Hg) \|_{L^{2} ( \mathbb{R}^{\vec{d}} )}\\
     &=  \sup_{\vec{x}_{\hat{\boldsymbol{v}}}}
    \|[T_{v_1,D_{v_1}},\ldots ,[T_{v_l,D_{v_l}} ,b ( \vec{x}^{\;0}_{\hat{\boldsymbol{v}}} ) ]\ldots]
    (f_H) \|_{L^{2} ( \mathbb{R}^{\vec{d}_{\boldsymbol{v}}})}\\
    & \geqslant  \|[T_{v_1,D_{v_1}},\ldots ,[T_{v_l,D_{v_l}} ,b ( \vec{x}^{\;0}_{\hat{\boldsymbol{v}}}) ]\ldots] (f_H) \|_{L^{2} ( \mathbb{R}^{\vec{d}_{\boldsymbol{v}}})}\\
    & \gtrsim  \| b ( \vec{x}^{\;0}_{\hat{\boldsymbol{v}}} ) \|_{\text{BMO}_{(v_1)\ldots(v_l)}(\mathbb{R}^{\vec{d}_{\boldsymbol{v}}})} =  \| b \|_{\text{BMO}_{\mathcal{I}}(\mathbb{R}^{\vec{d}})} .
  \end{split}
\]  
Note that the test function $g$ can be chosen with well distributed Fourier transform. 
Take any cones in the variables not indexed by $v_s$ and let $\vec{D}$ denote the tensor product of their projections. $f_T=P_{\vec{D}}g$. Notice that $$\|[T_{v_1,D_{v_1}},\ldots ,[T_{v_l,D_{v_l}} ,b ]\ldots] (
  f_H f_T ) \| \gtrsim
  \|[T_{v_1,D_{v_1}},\ldots ,[T_{v_l,D_{v_l}} ,b ]\ldots] (
  f_H g) \|$$ with constants depending upon how small the aperture of  the chosen cones is.
   Notice that the test function $f:=f_Hf_T$ has the Fourier support as required in part {\it{(2)}} of the statement of Lemma \ref{lemma_cone}. 
    
  Now build cones  $\vec{D}_s$ from the $D_{v_s}$ and the other chosen cones $D_k$ as well as operators ${T}_{s,\vec{D}_s}$. Notice that the commutators $[T_{v_1,D_{v_1}},\ldots ,[T_{v_l,D_{v_l}} ,b]\ldots]$
  and $[ T_{1,\vec{D}_{1}},\ldots , [ T_{l,\vec{D}_{l}} ,b ]\ldots ]$ reduce significantly when applied to a test function $f$ with Fourier support like ours.
  When the operators $T_{v_s,D_{v_s}}$ or any tensor product $T_{s,\vec{D}_{s}}$  
  fall directly on $f$, the contribution is zero due to opposing Fourier
  supports of the test function and the symbols of the operators. The only terms left in
   the commutators $[ T_{1,\vec{D}_{1}} ,\ldots , [ T_{l,\vec{D}_{v_l}},b ] \ldots](f) $ and $[ T_{v_1,D_{v_1}} ,\ldots, [ T_{v_l,D_{v_l}} ,b ] \ldots](f) $ have the form $\bigotimes_sT_{s,\vec{D}_{s}} (bf)$ and $\bigotimes_sT_{v_s,D_{v_s}} (bf)$ respectively. 
  
  By repeated use of Lemma \ref{lemma_realtoeplitz} we have the operator norm estimates for any
  symbol $b$, valid on the subspace of functions with Fourier support as described for $f$: $\| \bigotimes_sT_{s,\vec{D}_{s}} b \|_{2\to 2} = \| \bigotimes_sT_{v_s,D_{v_s}} b \|_{2\to 2} .$ We conclude that a normalised test function $f$ with Fourier support as described in the statement {\it{(2)}} of Lemma \ref{lemma_cone} exists, so that $\| \bigotimes_sT_{s,\vec{D}_{s}}
  (b f )\|_{2} \gtrsim \| b \|_{\text{BMO}_{\mathcal{I}}(\mathbb{R}^{\vec{d}})}$. In particular, we get the desired estimate in {\it{(1)}}.
  \end{proof}

It does not seem possible to pass directly to a lower commutator estimate for
tensor products of Riesz transforms from that for tensor products of cone
operators. Just using tensor products of operators adapted to cones merely gives us {\it some} lower bound where we are unable to control that a Riesz transform does appear in
every variable such as required in {\it (3)} of Theorem \ref{theorem_riesz}. The reason for this will
become clear as we advance in the argument. Instead of using operators $T_{s,
\vec{D}_s}$ directly, we will build upon them more general multi-parameter Journ\'e type
cone operators not of tensor product type that we now describe.

Let us explain the multiplier we
need for $i$ copies of $\mathbb{S}^{d - 1}$ when all dimensions are the same. We
will explain how to pass to the case of $i$ copies of varying dimension $d_k$ below. A picture illustrating a base case, a product of two 1-spheres, can be found at the end of this section.

For $0 < b < a < 1$, let $\varphi : [ - 1, 1] \rightarrow [ - 1, 1]$ be a
smooth function with $\varphi ( x) = 1$ when $a \leqslant x \leqslant 1$ \ and
$\varphi ( x) = 0$ when $b \geqslant x \geqslant 0$. And let $\varphi$ be odd,
meaning antisymmetric about $t = 0$. The function $\varphi$ gives rise to a
zonal function with pole $\xi_1$ on the first copy of $\mathbb{S}^{d - 1}$,
denoted by $C_1 ( \xi_1 ; \eta_1)$. This is the multiplier of a one-parameter
Calder\'on-Zygmund operator adapted to a cone $D ( \xi_1, r)$ for $r=\pi/2(1-a)$. 
For $i > 1$ we define $C_k ( \xi_1, \ldots, \xi_k ; \eta_1, \ldots,
\eta_k)$ for $1 < k \leqslant i$ inductively. In what follows, expectation is taken with
respect to traces of surface measure. When $\eta_i = \pm \xi_i$, then
conditional expectation is over a one-point set. 
\[ 
\begin{split}
&C_k ( \xi_1, \ldots, \xi_k ; \eta_1, \ldots, \eta_k)\\
 &=\mathbb{E}_{a_{k -
   1}} ( C_{k - 1} ( \xi_1, \ldots, a_{k - 1} ; \eta_1, \ldots, \eta_{k
   - 1}) \mid d ( a_{k - 1}, \xi_{k - 1}) = d ( \eta_k, \xi_k)) . 
   \end{split}
   \]
If the dimensions are not equal take $d=\max_{d_j}$ and imbed $\mathbb{S}^{d_j-1}$ into $\mathbb{S}^{d-1}$ by the map $\xi=(\xi_1,\ldots ,\xi_{d_j}) \mapsto (\xi_1,\ldots ,\xi_{d_j},0,\ldots ,0)$. 
Obtain in this manner the function $C_i$ and then restrict to the original number of variables when the dimension is smaller than $d$.

The multiplier $\vec{J}=C_i (\vec{\xi} ; \cdot)$ gives rise to a multi-parameter Calder\'on-Zygmund operator of convolution type (but not of tensor product type),
$\vec{T}_{\vec{J}}=\vec{T}_{C_i ( \vec{\xi} ; \cdot)}$. 
In fact, it is defined through
principal value convolution against a kernel $\vec{K}_{\vec{J}}=\vec{K}_{C_i(\vec{\xi};\cdot)} ( x_1, \ldots , x_i)$ such that
\[ \forall l : \int_{\alpha < | x_l | < \beta} \vec{K}_{\vec{J}} ( x_1, \ldots , x_i) d x_l = 0, 
\forall 0 < \alpha < \beta , x_j \in
   \mathbb{R}^{d_j} \text{ fixed } \forall j \neq l, \]
   
\[ | \frac{\partial^{| \vec{n} |}}{\partial x^{n_1}_1 \ldots \partial
   x^{n_i}_i} \vec{K}_{\vec{J}}( x_1, \ldots, x_i) | \leqslant
   A_{\vec{n}} | x_1 |^{- d_1 - n_1} \ldots | x_i |^{- d_i - n_i}, n_j \geqslant
   0. \]
   
   \medskip
   
This kind of operator is a special case of the more general, non-convolution
type discussed in Section \ref{section_upper}. It has many other nice features that will
facilitate our passage to Riesz transforms. One of them is its very special
representation in terms of homogeneous polynomials, the other one a lower commutator estimate in terms of the $\text{BMO}_{\mathcal{I}}$ norm.

\begin{lemma}\label{lemma_homogeneous polynomials}
  Let $C_i$ be a multiplier in $\bigotimes_{k=1}^{i}\mathbb{R}^{d_k}$ as described above, with any fixed direction and aperture. Let $m$ be an integer of order $d=\max{d_k}$. For any $\delta>0$, 
  the function $C_i$ has an approximation by a polynomial $C^{N}_i$ in the $\prod^i_{k = 1} d_k$
  variables $\{\prod_{k:1 \leqslant k  \leqslant i} \eta_{k, j_k}\mid  1
  \leqslant j_k \leqslant d_k \}$ so that $\|C_i-C^{N}_i\|_{\mathcal{C}^m(\mathbb{S}^{d_k-1})}<\delta$ in each variable separately.
  \end{lemma}

$\mathcal{C}^m$ indexes the norm of uniform convergence on functions that are $m$ times continuously differentiable.
On the space side, $C^{N}_i$ corresponds to an
  operator that is a polynomial in Riesz transforms of the variables
  $\bigotimes_k R_{k, j_k}$.

\begin{lemma}\label{lemma_lowerbd_Journe}
  We are in $\mathbb{R}^{\vec{d}}$ with partition $\mathcal{I}=(I_s)_{1\le s\le l}$. Let $\vec{\Upsilon}$ consist of vectors $\vec{\xi} = ( \xi_k)^t_{k = 1}$
  with $\xi_k \in \Upsilon_k$. Let $\vec{\Upsilon}^{(s)}$ consist of $\vec{\xi}^{( s)} = ( \xi_k)_{k \in I_s}$.
  Let us consider the class of Journ{\'e} type cone multipliers $\vec{J}_s = C_{i_s} (
  \vec{\xi}^{( s)} ; \cdot)$ of aperture $r_s$ with associated multi-parameter
  Calder\'on-Zygmund operators $\vec{T}_{s, \vec{J}_s}$. Then we have
  the two-sided estimate
  \[ \| b \|_{\text{BMO}_{\mathcal{I}} ( \mathbb{R}^{\vec{d}})} \lesssim
     \sup_{\vec{\xi} \in \vec{\Upsilon}} 
     \| [ \vec{T}_{1,\vec{J}_1}, \ldots, [ \vec{T}_{l, \vec{J}_l}, b] \ldots]
     \|_{L^2 ( \mathbb{R}^{\vec{d}}) \righttoleftarrow} \lesssim \| b \|_{\text{BMO}_{\mathcal{I}} (
     \mathbb{R}^{\vec{d}})}. \]
\end{lemma}

In order to proceed with the proof of these lemmas, we will use
some well known facts about zonal harmonics. 
Fix a pole $\xi \in \mathbb{S}^{d - 1}$. The zonal harmonic with pole $\xi$ of
degree $n$ is written as $Z^{( n)}_{\xi} ( \eta)$. With $t = \langle \xi, \eta
\rangle \in [ - 1, 1]$, one writes $Z^{( n)}_{\xi} ( \eta) = P_n ( t)$ where
$P_n$ is the Legendre polynomial of degree $n.$ It is common to suppress the
dependence on $d$ in the notation for $Z^{( n)}_{\xi}$ and $P_n$. 


$Z^{( n)}_{\xi}$ are reproducing for spherical harmonics of degree
$n$, $Y^{( n)}$.
When $Y^{( n)}$ is harmonic and homogeneous of degree $n$ with $Y^{( n)} (
\xi) = 1$ and $Y^{( n)} ( R \eta) = Y^{( n)} ( \eta)$ for any rotation $R \in
\mathcal{O} ( d)$ with $R \xi = \xi$, then $Y^{( n)} = Z_{\xi}^{( n)}$.

The lemma below will aid us in understanding the special form of the
functions $C_i$. 
\begin{lemma}\label{lemma_cosinus_average} Let $\xi_1, \xi_2 \in \mathbb{S}^{d-1}$. We have
  \[ 
  \begin{split}
  Z_{\xi_1}^{( n)} ( \eta_1)Z_{\xi_2}^{( n)} ( \eta_2) &=
  \mathbb{E}_{a_1} ( Z_{\eta_1}^{( n)} ( a_1) \mid d ( \xi_1, a_1) = d (
     \xi_2, \eta_2)) \\
     &=\mathbb{E}_{a_2} ( Z_{\eta_2}^{( n)} ( a_2) \mid d (
     \xi_2, a_2) = d ( \xi_1, \eta_1)). 
     \end{split}
     \]
\end{lemma}

\begin{proof}
  The first equality is a change of variable, thanks to symmetry of the zonal harmonic in its variables and invariance
  with respect to action of the measure preserving elements of the orthogonal
  group fixing poles $\xi_1$ or $\xi_2$, that we now detail. By a rotation in one of the spheres, assume $\xi_1=\xi_2=\xi$. Take a small ball $$B_{\xi,\eta_1}(a^0_2;\varepsilon_2)=\{a_2:d(a_2,a_2^0)<\varepsilon_2\}\cap \{a_2:d(a_2,\xi)=d(\eta_1,\xi)\}.$$ Note $\{a_2:d(a_2,\xi)=d(\eta_1,\xi)\}\sim\mathbb{S}^{d-2}$. Every $a_2\in B_{\xi,\eta_1}(a^0_2;\varepsilon_2)$ gives rise to a canonical orthogonal map $\sigma_{a_2}$ along geodesics in a scaled copy of $\mathbb{S}^{d-2}$. Lifted to $\mathbb{S}^{d-1}$, these are orthogonal maps fixing $\xi$. Let $\sigma^0$ fix $\xi$ and map $a_2^0$ to $\eta_1$. Let $a_1^0=\sigma^0(\eta_2)$. We observe that $\{\sigma^0 \sigma_{a_2}(\eta_2):a_2\in B_{\xi,\eta_1}(a^0_2;\varepsilon_2)\}=B_{\xi,\eta_2}(a^0_1;\varepsilon_1)$ with $\varepsilon_1$ so that $$\mathbb{P}(d(a_2,a_2^0)<\varepsilon_2 \mid d(\xi,a_2)=d(\xi,\eta_1))=\mathbb{P}(d(a_1,a_1^0)<\varepsilon_1 \mid d(\xi,a_1)=d(\xi,\eta_2)).$$ Together with the symmetry and the rotation property $Z^{(n)}_{\eta}(a)=Z^{(n)}_{a}(\eta)=Z^{(n)}_{\sigma(a)}(\sigma(\eta))$, we obtain the first equality.

  For fixed $a_1$, the function $Z_{\eta_1}^{( n)} ( a_1) =
  Z^{( n)}_{a_1} ( \eta_1)$ is a function harmonic in $\mathbb{R}^d$,
  $n$-homogeneous. These properties are preserved when taking expectation in
  $a_1$. So the expression
  $ \mathbb{E} ( Z^{( n)}_{\eta_1} ( a_1) \mid  d ( \xi_1, a_1) = d (
     \xi_2, \eta_2)) $
  remains harmonic (regarded as a function in $\mathbb{R}^d$),
  $n$-homogeneous. From the form
  $\mathbb{E} ( Z_{\eta_2}^{( n)} ( a_2) \mid d ( \xi_2, a_2) = d ( \xi_1,
     \eta_1)) $
  we learn that its restriction to $\mathbb{S}^{d - 1}$ depends only upon $d
  ( \xi_1, \eta_1)$. This implies that it is a constant multiple of the zonal
  harmonic with pole $\xi_1$. Exchanging the roles of $\eta_1$ and $\eta_2$
  gives
  \[ \mathbb{E} ( Z^{( n)}_{\eta_1} ( a_1) \mid d ( \xi_1, a_1) = d (
     \xi_2, \eta_2)) = c_n Z_{\xi_1}^{( n)} ( \eta_1) Z_{\xi_2}^{( n)} (
     \eta_2) . \]
  When assuming the normalization $Z^{( n)}_{\xi} ( \xi) = 1$ then $c_n = 1$. 

  This is a gernalisation of the classical symmetrising of the cosinus sum formula $1/2(\cos(x+y)+\cos(x-y))=\cos(x)\cos(y)$.

\end{proof}


\begin{proof}(of Lemma \ref{lemma_homogeneous polynomials})
It is well known that zonal harmonic series have convergence properties when
representing smooth zonal functions similar to that of the Fourier transform. For any given $m$ and sufficiently smooth $\varphi$ of the type
described above, then
\[ C_1 ( \xi_1 ; \eta_1) = \sum_n \varphi_n Z_{\xi_1}^{( n)} ( \eta_1) \]
where the convergence is $\mathcal{C}^m$-uniform. The degree of smoothness
required for $\varphi$ to obtain convergence in the $\mathcal{C}^m$ in the
above expression depends upon $m$ and the dimension $d$. For our purpose, we
choose $m \geqslant d$. 


Let us denote this function's representation of degree $N$ by a series of
zonal harmonics by $C^{( N)}_1 ( \xi_1 ; \eta_1)$.
\[ C^{( N)}_1 ( \xi_1 ; \eta_1) = \sum_{n \leqslant N} \varphi_n Z^{(
   n)}_{\xi_1} ( \eta_1) . \]
For every $\delta> 0$ there exists $N$ so that we have the estimate
\[ \| C^{( N)}_1 ( \xi_1 ; \eta_1) - C_1 ( \xi_1 ; \eta_1) \|_{\mathcal{C}^m (\mathbb{S}^{d_1-1})} < \delta. \]
In the case of $i$ copies of spheres, we define $C_i^{(N)}$ inductively in
the same manner as $C_i$.  Let us for the moment make all dimensions equal using the argument discussed above. So we set
\[ 
\begin{split}
&C_k^{( N)} ( \xi_1, \ldots, \xi_k ; \eta_1, \ldots, \eta_k)\\
   &=\mathbb{E}_{a_{k - 1}} ( C^{( N)}_{k - 1} ( \xi_1,  \ldots, a_{k -
   1} ; \eta_1, \ldots, \eta_{k - 1}) \mid d ( a_{k - 1}, \xi_{k - 1}) = d (
   \eta_k, \xi_k)). 
   \end{split}
   \]

 We claim the identity
  \begin{equation}\label{cone_identity} C^{( N)}_i ( \vec{\xi} ; \eta_1, \eta_2, \ldots, \eta_i) = \sum_{n
     \leqslant N} \varphi_n \prod_{k = 1}^i Z_{\xi_k}^{( n)} ( \eta_k) . 
     \end{equation}

  This is trivially true for $i = 1$. For $i > 1$
  induct on the number of parameters:
  \[
  \begin{split}
    &   C_i^{( N)} ( \vec{\xi} ; \eta_1, \ldots, \eta_i)\\
    & =  \mathbb{E}_{a_{i - 1}} ( C_{i - 1} ( \xi_1, \xi_2, \ldots, a_{i -
    1} ; \eta_1, \ldots, \eta_{i - 1}) \mid d ( a_{i - 1}, \xi_{i - 1}) = d (
    \eta_i, \xi_i))\\
    & =  \mathbb{E}_{a_{i - 1}} \left( \sum_{n \leqslant N} \varphi_n
    \prod_{k = 1}^{i - 1} Z_{\xi_k}^{( n)} ( \eta_k) \mid d ( a_{i - 1},
    \xi_{i - 1}) = d ( \eta_i, \xi_i) \right)\\
    & =  \sum_{n \leqslant N} \varphi_n \prod_{k = 1}^{i - 2} Z_{\xi_k}^{(
    n)} ( \eta_k) \mathbb{E}_{a_{i - 1}} ( Z_{\xi_{i - 1}}^{( n)} \mid d (
    a_{i - 1}, \xi_{i - 1}) = d ( \eta_i, \xi_i))\\
    & =  \sum_{n \leqslant N} \varphi_n \prod_{k = 1}^i Z_{\xi_k}^{( n)} (
    \eta_k) .
  \end{split}
  \]
  The first equality is the definition of $C^{( N)}_i$, the second one is the
  induction hypothesis and the last an application of Lemma \ref{lemma_cosinus_average}.

 It follows that neither $C_i$ nor $C^{(N)}_i$ depend on the order chosen in their definition and
  \[ C_i ( \vec{\xi} ; \eta_1, \ldots, \eta_i) = \sum_n \varphi_n \prod_{k =
     1}^i Z_{\xi_k}^{( n)} ( \eta_k) \]
  where the convergence is in $\mathcal{C}^m$ in each variable.

  Next, we study the terms arising in multipliers of the form $C^{( N)}_i$.
When all dimensions are equal, indeed, $\prod^i_{k = 1} Z_{\xi_k}^{( n)} ( \eta_k)$ has the important
property that, as a product of $n$-homogeneous polynomials, has only terms of
the form
\begin{equation*}
  \prod^i_{k = 1} \eta^{\alpha_k}_k = \prod_{k = 1}^i \left( \prod_{j_k=1}^d
  \eta^{\alpha_{k, j_k}}_{k, j_k} \right) \label{RieszproductFourier}
\end{equation*}
where $\eta_k \in \mathbb{S}^{d - 1}$ and $\alpha_k = ( \alpha_{k, j_k})$ are
multi-indices with $| \alpha_k | = \sum_{j_k} \alpha_{k, j_k} = n$ for all $k$. This
form is inherited by $C^{( N)}_i$ with varying $n$. It shows that $C^{(N)}_i$ is indeed a polynomial in the variables $\prod_{k=1}^{i}\eta_{k,j_k}$. 
When the dimensions $d_k$ are not equal, observe that by restricting back to the original  number of variables, we certainly lose harmonicity of the polynomials, 
but not $n$-homogeneity or the required form of our polynomials.
\end{proof}

 \begin{proof} (of Lemma \ref{lemma_lowerbd_Journe})
  By Lemma \ref{lemma_cone} we know that for each parameter $1 \leqslant s
  \leqslant l$ there exists a tensor product of cones $\vec{D}_s = \bigotimes_{k
  \in I_s} D ( \xi_k, r_k)$ with $r_s := \sum_{k \in I_s} r_k < \pi / 2$
  and $\xi_k \in \Upsilon_k$ \ and test functions $f_s$ supported as described
  in Lemma \ref{lemma_cone} part {\it (2)} so that
  \[ \| [ T_{1, \vec{D}_1}, \ldots, [ T_{l, \vec{D}_l}, b] \ldots] ( f) \|_2
     \gtrsim \| b \|_{\text{BMO}_{\mathcal{I}} ( \mathbb{R}^{\vec{d}})} \]
  where $f = \bigotimes^l_{s = 1} f_s$. 
  We make a remark about the apertures $r_s$. Let $d ( \cdot, \cdot)$ denote geodesic distance on $\mathbb{S}^{d - 1}$,
where antipodal points have distance $\pi$. Let $\vec{\xi}^{(s)}$ be the set of directions of $\vec{D}_s$. Remember that according to Lemma \ref{lemma_cone},
one component had a specific direction $\xi^{(s)}_v \in \Upsilon_v$ and possibly large aperture with 
$( 1 + \tau)r^{(s)}_v < \pi / 2$. Let us choose the other directions arbitrarily but with
apertures $r^{(s)}_{\hat{v}}$ small enough so that $( 1 + \tau)(r^{(s)}_v + ( i - 1) r^{(s)}_{\hat{v}} )< \pi
/ 2$. Now choose an aperture $r_s < \pi / 2$ so that $ ( 1 + \tau)( r^{(s)}_v + ( i
- 1) r^{(s)}_{\hat{v}}) < r_s < \pi / 2$.

  Writing $i_s = | I_s |$, we find
  Journ{\'e} type cone multipliers $\vec{J}_s = C_{i_s} ( \vec{\xi}^{(
  s)} ; \cdot)$ according to the construction above with center $\vec{\xi}^{(s)}$ and aperture $r_s$.   
  
  We are going to observe that $\vec{J}_s
  \equiv 1$ on ${\text{spt}} ( \vec{D}_s)$ and $\vec{J}_s \equiv - 1$ on
  the Fourier support of $f_s$. Let us drop the dependence on $s$ for the moment. 
  We  see in an inductive manner that $C_i ( \vec{\xi} ; \cdot)$ takes the
value 1 in a certain $\ell^1$ ball of radius $r < \pi / 2$ centered at
$\vec{\xi}$. We show that $$\sum_k d ( \xi_k, \eta_k) < r \Rightarrow C_i (
\vec{\xi}, \eta_1, \ldots, \eta_i) = 1.$$ When $i = 1$, the assertion is
obviously true: $d ( \xi_1, \eta_1) < r \Rightarrow C_1 ( \xi_1 ; \eta_1) = 1$
by the choice of $\varphi, r$ and definition of $C_1$. For $i > 1$, we proceed
by induction. Assume that $\sum^{i - 1}_{k = 1} d ( \xi_k, n_k) < r$ implies
$ C_{i - 1} ( \xi_1, \ldots, \xi_{i - 1} ; \eta_1, \ldots, \eta_{i -
1}) = 1$. 
Let us assume that $\sum^i_{k = 1} d ( \xi_k, \eta_k) < r$. Remembering the definition of $C_i(\vec{\xi};\cdot)$ we assume $d (
a_{i - 1}, \xi_{i - 1}) = d ( \eta_i, \xi_i)$. By the triangle inequality $\sum^{i
- 2}_{k = 1} d ( \xi_k, \eta_k) + d ( a_{i - 1}, \eta_{i - 1}) \leqslant
\sum^{i - 2}_{k = 1} d ( \xi_k, \eta_k) + d ( a_{i - 1}, \xi_{i - 1}) + d (
\xi_{i - 1}, \eta_{i - 1}) = \sum^i_{k = 1} d ( \xi_k, \eta_k) < r.$ So $$C_{i - 1} ( \xi_1, \xi_2,
\ldots, a_{i - 1} ; \eta_1, \ldots, \eta_{i - 1}) = 1$$ for all $a_{i - 1}$ relevant to the conditional expectation in the definition of $C_i(\vec{\xi};\cdot)$. The statement for $i$ follows. 

Since $C_i ( \vec{\xi} ; \cdot)$ does not depend
upon the order of the variables in its construction,
we are also able to see exactly as done above that when $\sigma_k = - 1$ for
exactly one choice of $k$, then $\sum_k d ( \sigma_k \xi_k, \eta_k) < r
\Rightarrow C_i ( \vec{\xi} ; \eta_1, \ldots, \eta_i) = - 1$.

   Consider associated multi-parameter
  Calder\'on-Zygmund operators $\vec{T}_{s, \vec{J}_s}$ and
  ${\vec{Id}}_s = \bigotimes_{k \in I_s} \text{Id}_k$ and
  $\text{Id}_k$ the identity on the $k^{\text{th}}$ variable. 
  Now
  \begin{eqnarray*}
    [ \vec{T}_{1, \vec{J}_1}, \ldots, [ \vec{T}_{l,
    \vec{J}_l}, b] \ldots] ( f) & = & [ \vec{T}_{1, \vec{J}_1}
    + {\vec{Id}}_1, \ldots, [ \vec{T}_{l, \vec{J}_l}
    + {\vec{Id}}_l, b] \ldots] ( f)\\
    & = & \bigotimes^l_{s = 1} ( \vec{T}_{s, \vec{J}_s} +
    {\vec{Id}}_s) ( b f)
  \end{eqnarray*}
  With $\| \bigotimes^l_{s = 1} ( \vec{T}_{s, \vec{J}_s} +
  {\vec{Id}}_s) ( b f) \|_2 \geqslant 
  \| \bigotimes^l_{s = 1} T_{s, \vec{D}_s} ( b f) \|_2$
  and  $\bigotimes^l_{s = 1} T_{s, \vec{D}_s} ( b f)= [ T_{1,
  \vec{D}_1}, \ldots, [ T_{l, \vec{D}_l}, b] \ldots] (  f)$
we get the desired lower bound on the Journ\'e commutator as claimed.
\end{proof}

Let us illustrate the base case $(\mathbb{S}^{1})^2$ by a picture. The picture is simplified in the sense that the odd function $\varphi$ above is replaced by an indicator function of an interval. 

\medskip

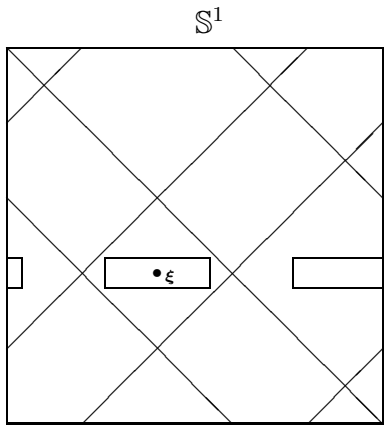
\begin{minipage}{0.35\textwidth}
\setlength{\unitlength}{1cm}
\begin{picture}(6,6)
\put(1,0){\line(1,1){4.0}}
\put(0,1){\line(1,1){4.0}}
\put(4,0){\line(1,1){1.0}}
\put(0,4){\line(1,1){1.0}}
\put(0,3){\line(1,-1){3.0}}
\put(0,5){\line(1,-1){5.0}}
\put(3,5){\line(1,-1){2.0}}
\put(2,2){\circle*{0.1}}
\put(1.3,1.8){\line(1,0){1.4}}
\put(3.8,1.8){\line(1,0){1.2}}
\put(0,1.8){\line(1,0){0.2}}
\put(1.3,2.2){\line(1,0){1.4}}
\put(3.8,2.2){\line(1,0){1.2}}
\put(0,2.2){\line(1,0){0.2}}
\put(1.3,1.8){\line(0,1){0.4}}
\put(3.8,1.8){\line(0,1){0.4}}
\put(2.7,1.8){\line(0,1){0.4}}
\put(0.2,1.8){\line(0,1){0.4}}
\put(0,0){\line(1,0){5}}
\put(0,5){\line(1,0){5}}
\put(0,0){\line(0,1){5}}
\put(5,0){\line(0,1){5}}
\put(5.2,2.5){$\mathbb{S}^1$}
\put(2.5,5.2){$\mathbb{S}^1$}
\put(2.1,1.9){\tiny{$\vec{\xi}$}}
\end{picture}

\end{minipage} \hfill
\begin{minipage}{0.6\textwidth}
Cone functions based on the oblique strips containing $\vec{\xi}$ are averaged. In the two-dimensional case, $\mathbb{S}^{1}$, expectation is over a one or two point set only. The rectangle around $\vec{\xi}$ with sides parallel to the axes representing $\mathbb{S}^1$ illustrates the support of the tensor product of cone operators with direction $\vec{\xi}$. The longer side is the aperture that arises from the Hankel part. The short sides can be chosen freely as they arise from the Toeplitz part and are chosen small so that the rectangle fits into the oblique square. The other small rectangle corresponds to the Fourier support of the test function $f$.
\end{minipage}

\medskip


\begin{proof}(of Theorem \ref{theorem_riesz})

In contrast to the Hilbert transform case, both lower bounds require
separate proofs. This is a notable difference that stems from the loss of orthogonal subspaces in conjunction with the special form of the Hilbert transform only seen in one variable. It does not seem possible to get a lower estimate {\it (3)}$\Rightarrow${\it (2)} directly.

{\it (1)}$\Leftrightarrow${\it (2)}. The upper bound {\it (1)}$\Rightarrow${\it (2)} is an easy consequence of the upper estimates of iterated commutators of single Riesz transforms. The lower 
bound {\it (2)}$\Rightarrow${\it (1)} follows from a standard fact on multipliers in combination with the main result
in \cite{LPPW}, the two-sided estimate for iterated commutators with
Riesz transforms, similar to the first arguments used in
\ref{lemma_cone}.

{\it (1)}$\Leftrightarrow${\it (3)}. The upper bound {\it (1)}$\Rightarrow${\it (3)} follows from the tensor product structure and use of the little product BMO norm (see also the remarks in section \ref{generalcase}). The lower bound {\it (3)}$\Rightarrow${\it (1)} uses the considerations
leading up to this proof: Suppressing again the dependence on $s$, 
we see that the multiplier $C_i$ is an odd, smooth, bounded function in
each $\eta_k$ when the other variables are fixed. 
Furthermore, since $\varphi$, written as a function of $t =
\langle \xi, \eta \rangle$ is odd with respect to $t = 0$, the above
series has $\varphi_n \neq 0$ at most when $n$ is odd and thus $Z_{\xi}^{(
n)}$ is odd. So $C^{(N)}_i$ is as a sum of odd functions odd.

We are now also ready to
see that $\vec{T}_{\vec{J}}$, the Journ\'e operator associated to the cone $\vec{J}=C_i(\vec{\xi};\cdot)$ as well as the operator associated to $C^{(N)}_i(\vec{\xi};\cdot)$ are
paraproduct free. In fact, applied to a test function $f = \bigotimes_k f_k$ with $f_k$ acting on the $k^{\text{th}}$ variable and where $f_l \equiv 1$ for some $l$ gives $\vec{T}_{\vec{J}} ( f) = 0$. To see
this, apply the multiplier $C^{(N)}_i ( \vec{\xi} ; \cdot)$ in the $l$ variable
(acting on 1) first, leaving the other Fourier variables fixed. The multiplier
function
\[ \eta_l \mapsto C^{(N)}_i ( \vec{\xi} ; \eta_1, \ldots, \eta_i) = \sum_{n
   \leqslant N} \varphi_n Z^{( n)}_{\xi_l} ( \eta_l) \prod_{k \neq l, k = 1}^i
   Z_{\xi_k}^{( n)} ( \eta_k) \]
is, as a sum of odd functions, odd on $\mathbb{S}^{d_l - 1}$, bounded by 1 and
uniformly smooth for all choices of $\eta_k$ with $k \neq l$. As such it gives
rise to a paraproduct free convolution type Calder\'on-Zygmund operator in the
$l$th variable whose values are multi-parameter multiplier operators.

Due to the convergence properties proved above, the difference
\[ C_i ( \vec{\xi} ; \cdot) - C^{( N)}_i ( \vec{\xi} ; \cdot) \]
gives rise to a paraproduct free Journ{\'e} operator with 
Calder\'on-Zygmund norm depending on $N$. This is seen by an application of an appropriate
version of the Marcinkievicz multiplier theorem.

By our stability result on Journ\'e commutators in section \ref{section_upper}, Corollary \ref{lemma_perturbation}, there exist for all $1\le s\le l$ integers $N_s$ so that $C^{( N_s)}_s (
\vec{\xi}_s ; \cdot)$ with $\xi_k \in \Upsilon_k$ is a characterizing set of
operators via commutators for $\text{BMO}_{\mathcal{I}} (
\mathbb{R}^{\vec{d}})$. This is a finite set of possibilities because of the universal choice of the $r_s$ and finiteness of the set $\vec{\Upsilon}$. Using the multi-parameter analog of the 
observation $[ A B, b] = A [ B, b] + [ A, b] B$ and the special form of the
$C^{( N_s)}_s ( \vec{\xi} ; \cdot)$, leaves us with the desired lower bound: Observe that when $[AB,b]$ has large $L^2$ norm then either $[A,b]$ or $[B,b]$ has a fair share of the norm. We use this argument finitely many times in a row for operators that are polynomials in tensor products of Riesz transforms $\bigotimes_{k\in I_s}R_{k,j_k}$. This finishes {\it (3)}$\Rightarrow${\it (1)}.
\end{proof}

We remark that there are two cases of dimension greater than 1, where the proof simplifies. In the case  of arbitrarily many copies of $\mathbb{R}^2$, one can work with the multiplicative structure of complex numbers and avoid the symmetrizing procedure to obtain cone functions with the appropriate polynomial approximations. 
 If the dimensions are arbitrary, but only tensor products of two Riesz transforms arise, one can avoid part of the construction above by using the addition formula for zonal harmonics. 

\section{Real variables: upper bounds}\label{section_upper}

In this section, we are interested in upper bounds for commutator norms by means of little product BMO norms of the symbol. In the case of the Hilbert transform, we have seen that these estimates, even in the iterated case, are straightforward. Other streamlined proofs exist for Hilbert or Riesz transforms when considering dyadic shifts of complexity one, see \cite{P}, \cite{PTV} and \cite{LPPW2}. When considering more general Calder\'on-Zygmund operators, the arguments required are more difficult, in each case. The first classical upper bound goes back to \cite{CRW}, where an estimate for one-parameter commutators with convolution type Calder\'on-Zygmund operators is given. Next, the text \cite{LPPW} includes a technical estimate for the multi-parameter case for such Calder\'on-Zygmund operators with a high enough degree of smoothness. This smoothness assumption was removed in \cite{DO} thanks to an approach using the representation formula for Calder\'on-Zygmund operators by means of infinite complexity dyadic shifts \cite{H}. This last proof also gives a control on the norm of the commutators which depends on the Calder\'on-Zygmund norm of the operators themselves, a fact we will employ later. Below, we give an estimate by little product BMO when the Calder\'on-Zygmund operators are of Journ\'e type and cannot be written as a tensor product. While this estimate is interesting in its own right, remember that it is also essential for our characterization result, the lower estimate, in section \ref{section_riesz}. 
The first generation of multi-parameter singular integrals that are not of tensor product type goes back to Fefferman \cite{F2} and was generalized by Journ\'e in \cite{J2} to the non-convolution type in the framework of his $T(1)$ theorem in this setting. Much later, Journ\'e's $T(1)$ theorem was revisited, for example in \cite{M}, \cite{O1}, \cite{O2}. See also \cite{M2} for some difficulties related to this subject. The references \cite{M} in the bi-parameter case and \cite{O2} in the general multi-parameter case include a representation formula by means of adapted, infinite complexity dyadic shifts. While these representation formulae look complicated, they have a feature very useful to us. `Locally', in a dyadic sense, they look as if they were of tensor product type, a feature we will exploit in the argument below. We start with the simplest bi-parameter case with no iterations and make comments about the generalization at the end of this section.

The class of bi-parameter singular integral operators treated in this section is that of any paraproduct free Journ\'e type operator (not necessarily a tensor product and not necessarily of convolution type) satisfying a certain weak boundedness property, which we define as follows:

\begin{definition}
A continuous linear mapping $T: C_0^\infty(\mathbb{R}^{n})\otimes C_0^\infty(\mathbb{R}^{m})\rightarrow [C_0^\infty(\mathbb{R}^{n})\otimes C_0^\infty(\mathbb{R}^{m})]'$ is called a \emph{paraproduct free bi-parameter Calder\'on-Zygmund operator} if the following conditions are satisfied:

1. $T$ is a Journ\'e type bi-parameter $\delta$-singular integral operator, i.e. there exists a pair $(K_1, K_2)$ of $\delta CZ$-$\delta$-standard kernels so that, for all $f_1,g_1\in C_0^\infty(\mathbb{R}^{n})$ and $f_2,g_2\in C_0^\infty(\mathbb{R}^{m})$,
\[
\pair{T(f_1\otimes f_2)}{g_1\otimes g_2}=\int f_1(y_1)\pair{K_1(x_1,y_1)f_2}{g_2}g_1(x_1)\,dx_1dy_1
\]
when $\text{spt}f_1\cap\text{spt}g_1=\emptyset$;
\[
\pair{T(f_1\otimes f_2)}{g_1\otimes g_2}=\int f_2(y_2)\pair{K_2(x_2,y_2)f_1}{g_1}g_2(x_2)\,dx_2dy_2
\]
when $\text{spt}f_2\cap\text{spt}g_2=\emptyset$.

2. $T$ satisfies the weak boundedness property $|\pair{T(\chi_I\otimes\chi_J)}{\chi_I\otimes\chi_J}|\lesssim |I||J|$, for any cubes $I\subset \mathbb{R}^n, J\in\mathbb{R}^m$.

3. $T$ is paraproduct free in the sense that $T(1\otimes \cdot)=T(\cdot\otimes 1)=T^*(1\otimes \cdot)=T^*(\cdot\otimes 1)=0$.

\end{definition}

Recall that a $\delta CZ$-$\delta$-standard kernel is a vector valued standard kernel taking values in the Banach space consisting of all Calder\'on-Zygmund operators. It is easy to see that an operator defined as above satisfies all the characterizing conditions in Martikainen \cite{M}, hence is $L^2$ bounded and can be represented as an average of bi-parameter dyadic shift operators together with dyadic paraproducts. Moreover, since $T$ is paraproduct free, one can conclude from observing the proof of Martikainen's theorem, that all the dyadic shifts in the representation are cancellative. 

The base case from which we pass to the general case below, is the following:

\begin{theorem}
Let $T$ be a paraproduct free bi-parameter Calder\'on-Zygmund operator, and $b$ be a little bmo function, there holds
\[
\|[b,T]\|_{L^2(\mathbb{R}^{n}\times\mathbb{R}^m)\righttoleftarrow }\lesssim \|b\|_{\text{bmo}(\mathbb{R}^n\times\mathbb{R}^m)},
\]
where the underlying constant depends only on the characterizing constants of $T$.
\end{theorem}

\begin{proof}

According to the discussion above, for any sufficiently nice functions $f,g$, one has the following representation:
\begin{equation}\label{repre}
\pair{Tf}{g}=C\mathbb{E}_{\omega_1}\mathbb{E}_{\omega_2}\sum_{i_1,j_1=0}^\infty\sum_{i_2,j_2=0}^\infty 2^{-\max(i_1,j_1)}2^{-\max(i_2,j_2)}\pair{S^{i_1j_1i_2j_2}f}{g},
\end{equation}
where expectation is with respect to a certain parameter of the dyadic grids. The dyadic shifts $S^{i_1j_1i_2j_2}$ are defined as
\[
\begin{split}
&S^{i_1j_1i_2j_2}f\\
&:=\sum_{K_1\in\mathcal{D}_1}\sum_{\substack{I_1,J_1\subset K_1, I_1,J_1\in\mathcal{D}_1\\\ell(I_1)=2^{-i_1}\ell(K_1)\\\ell(J_1)=2^{-j_1}\ell(K_1)}}\sum_{K_2\in\mathcal{D}_2}\sum_{\substack{I_2,J_2\subset K_2, I_2,J_2\in\mathcal{D}_2\\\ell(I_2)=2^{-i_2}\ell(K_2)\\\ell(J_2)=2^{-j_2}\ell(K_2)}}a_{I_1J_1K_1I_2J_2K_2}\pair{f}{h_{I_1}\otimes h_{I_2}}h_{J_1}\otimes h_{J_2}\\
&=\sum_{K_1}\sum_{I_1,J_1\subset K_1}^{(i_1,j_1)}\sum_{K_2}\sum_{I_2,J_2\subset K_2}^{(i_2,j_2)}a_{I_1J_1K_1I_2J_2K_2}\pair{f}{h_{I_1}\otimes h_{I_2}}h_{J_1}\otimes h_{J_2}.
\end{split}
\]

The coefficients above satisfy $a_{I_1J_1K_1I_2J_2K_2}\leq \frac{\sqrt{|I_1||J_1||I_2||J_2|}}{|K_1||K_2|}$, which also guarantees the normalization $\|S^{i_1j_1i_2j_2}\|_{L^2\rightarrow L^2}\leq 1$. Moreover, since $T$ is paraproduct free, all the Haar functions appearing above are cancellative.

It thus suffices to show that for any dyadic grids $\mathcal{D}_1,\mathcal{D}_2$ and fixed $i_1,j_1,i_2,j_2\in\mathbb{N}$, one has
\begin{equation}\label{shift}
\|[b,S^{i_1j_1i_2j_2}]f\|_{L^2}\lesssim (1+\max(i_1,j_1))(1+\max(i_2,j_2))\|b\|_{\text{bmo}}\|f\|_{L^2},
\end{equation}
as the decay factor $2^{-\max(i_1,j_1)}, 2^{-\max(i_2,j_2)}$ in (\ref{repre}) will guarantee the convergence of the series.

To see (\ref{shift}), one decomposes $b$ and a $L^2$ test function $f$ using Haar bases:
\[
[b,S^{i_1j_1i_2j_2}]f=\sum_{I_1,I_2}\sum_{J_1,J_2}\pair{b}{h_{I_1}\otimes h_{I_2}}\pair{f}{h_{J_1}\otimes h_{J_2}}[h_{I_1}\otimes h_{I_2},S^{i_1j_1i_2j_2}]h_{J_1}\otimes h_{J_2}.
\]

A similar argument to that in \cite{DO} implies that $[h_{I_1}\otimes h_{I_2},S^{i_1j_1i_2j_2}]h_{J_1}\otimes h_{J_2}$ is nonzero only if $I_1\subset J_1^{(i_1)}$ or $I_2\subset J_2^{(i_2)}$, where $J_1^{(i_1)}$ denotes the $i_1$-th dyadic ancestor of $J_1$, similarly for $J_2^{(i_2)}$. Hence, the sum can be decomposed into three parts: $I_1\subset J_1^{(i_1)}$ and $I_2\subset J_2^{(i_2)}$ (regular), $I_1\subset J_1^{(i_1)}$ and $I_2\supsetneq J_2^{(i_2)}$, $I_1\supsetneq J_1^{(i_1)}$ and $I_2\subset J_2^{(i_2)}$ (mixed). \\

{\it 1) Regular case:}

Following \cite{DO} one can decompose the arising sum into sums of classical bi-parameter dyadic paraproducts $B_0(b,f)$ and its slightly revised version $B_{k,l}(b,f)$: for any integers $k,l\geq 0$, $B_{k,l}$ is the bi-parameter dyadic paraproduct  defined as
\[
B_{k,l}(b,f)=\sum_{I,J}\beta_{IJ}\pair{b}{h_{I^{(k)}}\otimes u_{J^{(l)}}}\pair{f}{h_I^{\varepsilon_1}\otimes u_J^{\varepsilon_2}}h_I^{\varepsilon_1'}\otimes u_J^{\varepsilon_2'}|I^{(k)}|^{-1/2}|J^{(l)}|^{-1/2},
\]
where $\beta_{IJ}$ is a sequence satisfying $|\beta_{IJ}|\leq 1$. When $k>0$, all Haar functions in the first variable are cancellative, while when $k=0$, there is at most one of $h_I^{\varepsilon_1}, h_I^{\varepsilon_1'}$ being noncancellative. The same assumption goes for the second variable. Observe that when $k=l=0$, $B_{k,l}$ becomes the classical paraproduct $B_0$. It is proved in \cite{DO} that $$\|B_{k,l}(b,f)\|_{L^2}\lesssim\|b\|_{\text{BMO}}\|f\|_{L^2}$$ with a constant independent of $k,l$ and the product BMO norm on the right hand side.

Then since little bmo functions are contained in product BMO, this part can be controlled. More specifically, write
\[
\begin{split}
[b,S^{i_1j_1i_2j_2}]f&=\sum_{I_1,I_2}\sum_{J_1,J_2}\pair{b}{h_{I_1}\otimes h_{I_2}}\pair{f}{h_{J_1}\otimes h_{J_2}}h_{I_1}\otimes h_{I_2}S^{i_1j_1i_2j_2}(h_{J_1}\otimes h_{J_2})\\
&\quad-\sum_{I_1,I_2}\sum_{J_1,J_2}\pair{b}{h_{I_1}\otimes h_{I_2}}\pair{f}{h_{J_1}\otimes h_{J_2}}S^{i_1j_1i_2j_2}(h_{I_1}h_{J_1}\otimes h_{I_2}h_{J_2})\\
&=:I+II,
\end{split}
\]
then one can estimate term I and II separately. According to the definition of dyadic shifts, term I is equal to
\[
\begin{split}
&\sum_{J_1,J_2}\sum_{I_1: I_1\subset J_1^{(i_1)}}\sum_{I_2: I_2\subset J_2^{(i_2)}}\pair{b}{h_{I_1}\otimes h_{I_2}}\pair{f}{h_{J_1}\otimes h_{J_2}}h_{I_1}\otimes h_{I_2}\cdot\\
&\qquad  \left(\vphantom{\sum}\right.\sum_{\substack{J_1':J_1'\subset J_1^{(i_1)}\\\ell(J_1')=2^{i_1-j_1}\ell(J_1)}}\sum_{\substack{J_2':J_2'\subset J_2^{(i_2)}\\\ell(J_2')=2^{i_2-j_2}\ell(J_2)}}a_{J_1J_1'J_1^{(i_1)}J_2J_2'J_2^{(i_2)}}h_{J_1'}\otimes h_{J_2'}\left.\vphantom{\sum}\right)\\
&=\sum_{K_1,K_2}\sum_{J_1: J_1\subset K_1}^{(i_1)}\sum_{J_2:J_2\subset K_2}^{(i_2)}\sum_{I_1:I_1\subset K_1}\sum_{I_2:I_2\subset K_2}\pair{b}{h_{I_1}\otimes h_{I_2}}\pair{f}{h_{J_1}\otimes h_{J_2}}h_{I_1}\otimes h_{I_2}\cdot\\
&\qquad \left(\vphantom{\sum}\right.\sum_{J_1':J_1'\subset K_1}^{(j_1)}\sum_{J_2':J_2'\subset K_2}^{(j_2)}a_{J_1J_1'K_1J_2J_2'K_2}h_{J_1'}\otimes h_{J_2'}\left.\vphantom{\sum}\right)\\
&=\sum_{I_1,I_2}\pair{b}{h_{I_1}\otimes h_{I_2}}h_{I_1}\otimes h_{I_2}\sum_{\substack{K_1\supset I_1\\K_2\supset I_2}}
\sum_{J_1,J_1'\subset K_1}^{(i_1,j_1)}\sum_{J_2,J_2'\subset K_2}^{(i_2,j_2)}a_{J_1J_1'K_1J_2J_2'K_2}\pair{f}{h_{J_1}\otimes h_{J_2}}h_{J_1'}\otimes h_{J_2'}\\
&=\sum_{I_1,I_2}\pair{b}{h_{I_1}\otimes h_{I_2}}h_{I_1}\otimes h_{I_2}\sum_{J_1':J_1'^{(j_1)}\supset I_1}\sum_{J_2':J_2'^{(j_2)}\supset I_2}\pair{S^{i_1j_1i_2j_2}f}{h_{J_1'}\otimes h_{J_2'}}h_{J_1'}\otimes h_{J_2'}.
\end{split}
\]
Because of the supports of Haar functions, the inner sum above can be further decomposed into four parts, where
\[
\begin{split}
&I=\sum_{I_1,I_2}\sum_{J_1'\supsetneq I_1}\sum_{J_2'\supsetneq I_2},\quad
II=\sum_{I_1,I_2}\sum_{J_1'\supsetneq I_1}\sum_{J_2': J_2'\subset I_2\subset J_2'^{(j_2)}}\\
&III=\sum_{I_1,I_2}\sum_{J_1':J_1'\subset I_1\subset J_1'^{(j_1)}}\sum_{J_2'\supsetneq I_2},\quad
IV=\sum_{I_1,I_2}\sum_{J_1':J_1'\subset I_1\subset J_1'^{(j_1)}}\sum_{J_2': J_2'\subset I_2\subset J_2'^{(j_2)}}.
\end{split}
\]
Hence, using the same technique as in \cite{DO}, one has
\[
I=\sum_{I_1,I_2}\pair{b}{h_{I_1}\otimes h_{I_2}}\pair{S^{i_1j_1i_2j_2}f}{h_{J_1'}^1\otimes h_{J_2'}^1}h_{I_1}\otimes h_{I_2}|I_1|^{-1/2}|I_2|^{-1/2},
\]
which is a bi-parameter paraproduct $B_0(b,f)$. Moreover, one has
\[
\begin{split}
II&=\sum_{I_1,I_2}\pair{b}{h_{I_1}\otimes h_{I_2}}h_{I_1}\otimes h_{I_2}\sum_{J_2':J_2'\subset I_2\subset J_2'^{(j_2)}}\pair{S^{i_1j_1i_2j_2}f}{h_{I_1}^1\otimes h_{J_2'}}|I_1|^{-1/2}h_{J_2'}\\
&=\sum_{l=0}^{j_2}\sum_{I_1,J_2'}\beta_{J_2'}\pair{b}{h_{I_1}\otimes h_{J_2'^{(l)}}}\pair{S^{i_1j_1i_2j_2}f}{h_{I_1}^1\otimes h_{J_2'}}h_{I_1}\otimes h_{J_2'}|I_1|^{-1/2}|J_2'^{(l)}|^{-1/2}\\
&=\sum_{l=0}^{j_2}B_{0l}(b,S^{i_1j_1i_2j_2}f),
\end{split}
\]
where constants $\beta_{J_2'}\in\{1,-1\}$, and $B_{0l}$ are the generalized bi-parameter paraproducts of type $(0,l)$ defined in \cite{DO} whose $L^2\rightarrow L^2$ operator norm is uniformly bounded by $\|b\|_{\text{BMO}}$ product BMO.
Similarly, one can show that
\[
III=\sum_{k=0}^{j_1}B_{k0}(b,S^{i_1j_1i_2j_2}f),\quad IV=\sum_{k=0}^{j_1}\sum_{l=0}^{j_2}B_{kl}(b,S^{i_1j_1i_2j_2}f).
\]

Since $\|b\|_{\text{BMO}}\lesssim \|b\|_{\text{bmo}}$, all the forms above are $L^2$ bounded. This completes the discussion of term I.

To get an estimate of term II, we need to decompose it into finite linear combinations of $S^{i_1j_1i_2j_2}(B_{kl}(b,f))$. By linearity, one can write $S^{i_1j_1i_2j_2}$ on the outside from the beginning, and we will only look at the inside sum. One splits for example the sum regarding the first variable into three parts: $I_1\subsetneq J_1$, $I_1=J_1$, $J_1\subsetneq I_1\subset J_1^{(i_1)}$. If we split the second variable as well, there are nine mixed parts, and it's not hard to show that each of them can be represented as a finite sum of $B_{kl}(b,f)$. We omit the details. \\

{\it 2) Mixed case.}
Let's call the second and the third `mixed' parts, and as the two are symmetric, it suffices to look at the second one, i.e. $I_1\subset J_1^{(i_1)},I_2\supsetneq J_2^{(i_2)}$. In the first variable, we still have the old case $I_1\subset J_1^{(i_1)}$ that appeared in \cite{DO}, so morally speaking, we only need to nicely play around with the stronger little bmo norm to handle the second variable. For any fixed $I_1,J_1,I_2,J_2$, since $I_2\supsetneq J_2^{(i_2)}$, the definition of dyadic shifts implies that 
\[
h_{I_1}\otimes h_{I_2}S^{i_1j_1i_2j_2}(h_{J_1}\otimes h_{J_2})=h_{I_1}S^{i_1j_1i_2j_2}(h_{J_1}\otimes h_{I_2}h_{J_2})
\]
and
\[
S^{i_1j_1i_2j_2}(h_{i_1}h_{J_1}\otimes h_{I_2}h_{J_2})=h_{I_2}S^{i_1j_1i_2j_2}(h_{I_1}h_{J_1}\otimes h_{J_2}).
\]

Hence, we still have cancellation in the second variable, which converts the mixed case to
\[
\begin{split}
&\sum_{I_1\subset J_1^{(i_1)}}\sum_{I_2\supsetneq J_2^{(i_2)}}\pair{b}{h_{I_1}\otimes h_{I_2}}\pair{f}{h_{J_1}\otimes h_{J_2}}[h_{I_1},S^{i_1j_1i_2j_2}](h_{J_1}\otimes h_{I_2}h_{J_2})\\
&=\sum_{I_1\subset J_1^{(i_1)}}\sum_{J_2}\pair{f}{h_{J_1}\otimes h_{J_2}}[h_{I_1},S^{i_1j_1i_2j_2}](h_{J_1}\otimes\sum_{I_2\supsetneq J_2^{(i_2)}}\pair{b}{h_{I_1}\otimes h_{I_2}}h_{I_2}h_{J_2})\\
&=\sum_{I_1\subset J_1^{(i_1)}}\sum_{J_2}\pair{f}{h_{J_1}\otimes h_{J_2}}[h_{I_1},S^{i_1j_1i_2j_2}](h_{J_1}\otimes \pair{b}{h_{I_1}\otimes h_{J_2^{(i_2)}}^1}h_{J_2^{(i_2)}}^1h_{J_2})\\
&=\sum_{I_1\subset J_1^{(i_1)}}\sum_{J_2}\pair{b}{h_{I_1}\otimes h_{J_2^{(i_2)}}^1}|J_2^{(i_2)}|^{-1/2}\pair{f}{h_{J_1}\otimes h_{J_2}}[h_{I_1},S^{i_1j_1i_2j_2}](h_{J_1}\otimes h_{J_2})\\
&=\sum_{I_1\subset J_1^{(i_1)}}\sum_{J_2}\pair{\ave{b}_{J_2^{(i_2)}}}{h_{I_1}}_1\pair{f}{h_{J_1}\otimes h_{J_2}}[h_{I_1},S^{i_1j_1i_2j_2}](h_{J_1}\otimes h_{J_2}),
\end{split}
\]
where $\ave{b}_{J_2^{(i_2)}}$ denotes the average value of $b$ on $J_2^{(i_2)}$, which is a function of only the first variable. 

In the following, we will once again estimate the first term and second term of the commutator separately, and the $L^2$ norm of each of them will be proved to be bounded by $\|b\|_{\text{bmo}}\|f\|_{L^2}$.

a) First term.

By definition of the dyadic shift, the first term is equal to
\[
\begin{split}
&\sum_{I_1\subset J_1^{(i_1)}}\sum_{J_2}\pair{\ave{b}_{J_2^{(i_2)}}}{h_{I_1}}_1h_{I_1}\pair{f}{h_{J_1}\otimes h_{J_2}}\cdot\\
&\qquad\left(\vphantom{\sum}\right.\sum_{\substack{J_1'\subset J_1^{(i_1)}\\\ell(J_1')=2^{i_1-j_1}\ell(J_1)}}\sum_{\substack{J_2'\subset J_2^{(i_2)}\\\ell(J_2')=2^{i_2-j_2}\ell(J_2)}}a_{J_1J_1'J_1^{(i_1)}J_2J_2'J_2^{(i_2)}}h_{J_1'}\otimes h_{J_2'}\left.\vphantom{\sum}\right),
\end{split}
\]
which by reindexing $K_1:=J_1^{(i_1)}$ is the same as
\[
\begin{split}
&\sum_{I_1,J_2}\pair{\ave{b}_{J_2^{(i_2)}}}{h_{I_1}}_1h_{I_1}\cdot \\
&\;\; \cdot \sum_{K_1:K_1\supset I_1}\sum_{J_1\subset K_1}^{(i_1)}\sum_{J_1'\subset K_1}^{(j_1)}\sum_{J_2'\subset J_2^{(i_2)}}^{(j_2)}a_{J_1J_1'K_1J_2J_2'J_2^{(i_2)}}\pair{f}{h_{J_1}\otimes h_{J_2}}h_{J_1'}\otimes h_{J_2'}\\
&=\sum_{I_1,J_2}\pair{\ave{b}_{J_2^{(i_2)}}}{h_{I_1}}_1h_{I_1}\sum_{J_1':J_1'^{(j_1)}\supset I_1}h_{J_1'}\otimes\pair{S^{i_1j_1i_2j_2}(\pair{f}{h_{J_2}}_2\otimes h_{J_2})}{h_{J_1'}}_1,
\end{split}
\]
where the inner sum is the orthogonal projection of the image of $\pair{f}{h_{J_2}}_2\otimes h_{J_2}$ under $S^{i_1j_1i_2j_2}$ onto the span of $\{h_{J_1'}\}$ such that $J_1'^{(j_1)}\supset I_1$. Taking into account the supports of the Haar functions in the first variable, one can further split the sum into two parts where
\[
I:=\sum_{J_2}\sum_{I_1\subsetneq J_1'},\quad II:=\sum_{J_2}\sum_{J_1'\subset I_1\subset J_1'^{(j_1)}}.
\]
Summing over $J_1'$ first implies that
\[
\begin{split}
I&=\sum_{J_2}\sum_{I_1}\pair{\ave{b}_{J_2^{(i_2)}}}{h_{I_1}}_1h_{I_1}\left(h_{I_1}^1\otimes \pair{S^{i_1j_1i_2j_2}(\pair{f}{h_{J_2}}_2\otimes h_{J_2})}{h_{I_1}^1}_1\right)\\
&=:\sum_{J_2}B_0(\ave{b}_{J_2^{(i_2)}},S^{i_1j_1i_2j_2}(\pair{f}{h_{J_2}}_2\otimes h_{J_2}))
\end{split}
\]
where $B_0(b,f):=\sum_{I}\pair{b}{h_I}\pair{f}{h_I^1}h_I|I|^{-1/2}$ is a classical one-parameter paraproduct in the first variable. Note that its $L^2$ norm is bounded by $\|b\|_{\text{BMO}}\|f\|_{L^2}$.
Moreover, according to the definition of $S^{i_1j_1i_2j_2}$, for any fixed $J_2$
\[
S^{i_1j_1i_2j_2}(\pair{f}{h_{J_2}}_2\otimes h_{J_2})=\sum_{J_2':J_2'^{(j_2)}=J_2^{(i_2)}}\pair{S^{i_1j_1i_2j_2}(\pair{f}{h_{J_2}}_2\otimes h_{J_2})}{h_{J_2'}}_2\otimes h_{J_2'}.
\]
In other words, $S^{i_1j_1i_2j_2}(\pair{f}{h_{J_2}}_2\otimes h_{J_2})$ only lives on the span of $\{h_{J_2'}:J_2'^{(j_2)}=J_2^{(i_2)}\}$. Hence, by linearity there holds
\[
\begin{split}
I&=\sum_{J_2}\sum_{J_2':J_2'^{(j_2)}=J_2^{(i_2)}}B_0\big(\ave{b}_{J_2^{(i_2)}},\pair{S^{i_1j_1i_2j_2}(\pair{f}{h_{J_2}}_2\otimes h_{J_2})}{h_{J_2'}}_2\big)\otimes h_{J_2'}\\
&=\sum_{J_2'}\left(\vphantom{\sum}\right.B_0\big(\ave{b}_{J_2'^{(j_2)}},\pair{S^{i_1j_1i_2j_2}(\sum_{J_2:J_2^{(i_2)}=J_2'^{(j_2)}}\pair{f}{h_{J_2}}_2\otimes h_{J_2})}{h_{J_2'}}_2\big)\left.\vphantom{\sum}\right)\otimes h_{J_2'}.
\end{split}
\]
Thus, orthogonality in the second variable implies that
\[
\begin{split}
&\|I\|_{L^2(\mathbb{R}^n\times\mathbb{R}^m)}^2\\
&=\sum_{J_2'}\|B_0\big(\ave{b}_{J_2'^{(j_2)}},\pair{S^{i_1j_1i_2j_2}(\sum_{J_2:J_2^{(i_2)}=J_2'^{(j_2)}}\pair{f}{h_{J_2}}_2\otimes h_{J_2})}{h_{J_2'}}_2\big)\|_{L^2(\mathbb{R}^n)}^2\\
&\lesssim\sum_{J_2'}\|\ave{b}_{J_2'^{(j_2)}}\|_{\text{BMO}(\mathbb{R}^n)}^2\|\pair{S^{i_1j_1i_2j_2}(\sum_{J_2:J_2^{(i_2)}=J_2'^{(j_2)}}\pair{f}{h_{J_2}}_2\otimes h_{J_2})}{h_{J_2'}}_2\|_{L^2(\mathbb{R}^n)}^2.
\end{split}
\]
Observing that $\|\ave{b}_{J_2'^{(j_2)}}\|_{\text{BMO}(\mathbb{R}^n)}\leq \ave{\|b\|_{\text{BMO}(\mathbb{R}^n)}}_{J_2'^{(j_2)}}\leq\|b\|_{\text{bmo}}$, one has
\[
\begin{split}
&\leq \|b\|_{\text{bmo}}^2\sum_{J_2'}\|\pair{S^{i_1j_1i_2j_2}(\sum_{J_2:J_2^{(i_2)}=J_2'^{(j_2)}}\pair{f}{h_{J_2}}_2\otimes h_{J_2})}{h_{J_2'}}_2\|_{L^2(\mathbb{R}^n)}^2\\
&=\|b\|_{\text{bmo}}^2\|\sum_{J_2'}\pair{S^{i_1j_1i_2j_2}(\sum_{J_2:J_2^{(i_2)}=J_2'^{(j_2)}}\pair{f}{h_{J_2}}_2\otimes h_{J_2})}{h_{J_2'}}_2\otimes h_{J_2'}\|_{L^2(\mathbb{R}^n\times\mathbb{R}^m)}^2.
\end{split}
\]
Note that the sum in the $L^2$ norm is in fact very simple:
\[
\begin{split}
&\sum_{J_2'}\pair{S^{i_1j_1i_2j_2}(\sum_{J_2:J_2^{(i_2)}=J_2'^{(j_2)}}\pair{f}{h_{J_2}}_2\otimes h_{J_2})}{h_{J_2'}}_2\otimes h_{J_2'}\\
&=\sum_{J_2}\sum_{J_2':J_2'^{(j_2)}=J_2^{(i_2)}}\pair{S^{i_1j_1i_2j_2}(\pair{f}{h_{J_2}}_2\otimes h_{J_2})}{h_{J_2'}}_2\otimes h_{J_2'}\\
&=\sum_{J_2}S^{i_1j_1i_2j_2}(\pair{f}{h_{J_2}}_2\otimes h_{J_2})=S^{i_1j_1i_2j_2}(f).
\end{split}
\]
Hence, the uniform boundedness of the $L^2\rightarrow L^2$ operator norm of dyadic shifts implies that
\[
\|I\|_{L^2(\mathbb{R}^n\times\mathbb{R}^m)}^2\lesssim\|b\|_{\text{bmo}}^2\|f\|_{L^2(\mathbb{R}^n\times\mathbb{R}^m)}^2.
\]
In order to handle II, we split it into a finite sum depending on the levels of $I_1$ upon $J_1'$, which leads to
\[
\begin{split}
II&=\sum_{k=0}^{j_1}\sum_{J_2}\sum_{J_1'}\pair{\ave{b}_{J_2^{(i_2)}}}{h_{J_1'^{(k)}}}_1h_{J_1'^{(k)}}h_{J_1'}\otimes\pair{S^{i_1j_1i_2j_2}(\pair{f}{h_{J_2}}_2\otimes h_{J_2})}{h_{J_1'}}_1\\
&=\sum_{k=0}^{j_1}\sum_{J_2}\sum_{J_1'}\beta_{J_1',k}|J_1'^{(k)}|^{-1/2}\pair{\ave{b}_{J_2^{(i_2)}}}{h_{J_1'^{(k)}}}_1h_{J_1'}\otimes\pair{S^{i_1j_1i_2j_2}(\pair{f}{h_{J_2}}_2\otimes h_{J_2})}{h_{J_1'}}_1\\
&=:\sum_{k=0}^{j_1}\sum_{J_2}B_k(\ave{b}_{J_2^{(i_2)}},S^{i_1j_1i_2j_2}(\pair{f}{h_{J_2}}_2\otimes h_{J_2})),
\end{split}
\]
where $B_k(b,f):=\sum_I\beta_{I,k}\pair{b}{h_{I^{(k)}}}\pair{f}{h_I}h_I|I^{(k)}|^{-1/2}$ is a generalized one-parameter paraproduct studied in \cite{DO}, whose $L^2$ norm is uniformly bounded by $\|b\|_{\text{BMO}}\|f\|_{L^2}$, independent of $k$ and the coefficients $\beta_{I,k}\in\{1,-1\}$.
Then one can proceed as in part I to conclude that
\[
\|II\|_{L^2(\mathbb{R}^n\times\mathbb{R}^m)}\lesssim (1+j_1)\|b\|_{\text{bmo}}\|f\|_{L^2(\mathbb{R}^n\times\mathbb{R}^m)},
\]
which together with the estimate for part I implies that
\[
\|\text{First term}\|_{L^2(\mathbb{R}^n\times\mathbb{R}^m)}\lesssim (1+j_1)\|b\|_{\text{bmo}}\|f\|_{L^2(\mathbb{R}^n\times\mathbb{R}^m)}.
\]

b) Second term.

As the second term by linearity is the same as
\[
S^{i_1j_1i_2j_2}\left(\vphantom{\sum}\right.\sum_{J_2}\sum_{I_1\subset J_1^{(i_1)}}\pair{\ave{b}_{J_2^{(i_2)}}}{h_{I_1}}_1\pair{f}{h_{J_1}\otimes h_{J_2}}h_{I_1}h_{J_1}\otimes h_{J_2}\left.\vphantom{\sum}\right),
\]
the $L^2\rightarrow L^2$ boundedness of the shift implies that it suffices to estimate the $L^2$ norm of the term inside the parentheses. Since $I_1\cap J_1\neq\emptyset$, one can further split the sum into two parts: 
\[
I:=\sum_{J_2}\sum_{I_1\subsetneq J_1},\quad II:=\sum_{J_2}\sum_{J_1\subset I_1\subset J_1^{(i_1)}}.
\]
Summing over $J_1$ first implies that
\[
\begin{split}
I&=\sum_{J_2}\sum_{I_1}\pair{\ave{b}_{J_2^{(i_2)}}}{h_{I_1}}_1\pair{f}{h_{I_1}^1\otimes h_{J_2}}h_{I_1}h_{I_1}^1\otimes h_{J_2}\\
&=:\sum_{J_2}B_0(\ave{b}_{J_2^{(i_2)}},\pair{f}{h_{J_2}}_2)\otimes h_{J_2},
\end{split}
\]
where $B_0(b,f):=\sum_I\pair{b}{h_I}\pair{f}{h_I^1}h_I|I|^{-1/2}$ is a classical one-parameter paraproduct in the first variable. Hence,
\[
\begin{split}
\|I\|_{L^2(\mathbb{R}^n\times\mathbb{R}^m)}^2&=\sum_{J_2}\|B_0(\ave{b}_{J_2^{(i_2)}},\pair{f}{h_{J_2}}_2)\|_{L^2(\mathbb{R}^n)}^2\\
&\lesssim\sum_{J_2}\|\ave{b}_{J_2^{(i_2)}}\|_{\text{BMO}(\mathbb{R}^n)}^2\|\pair{f}{h_{J_2}}_2\|_{L^2(\mathbb{R}^n)}^2\\
&\leq\|b\|_{\text{bmo}}^2\sum_{J_2}\|\pair{f}{h_{J_2}}_2\|_{L^2(\mathbb{R}^n)}^2=\|b\|_{\text{bmo}}^2\|f\|_{L^2(\mathbb{R}^n\times\mathbb{R}^m)}^2.
\end{split}
\]
For part II, note that it can be decomposed as
\[
\begin{split}
II&=\sum_{k=0}^{i_1}\sum_{J_2}\sum_{J_1}\pair{\ave{b}_{J_2^{(i_2)}}}{h_{J_1^{(k)}}}_1\pair{f}{h_{J_1}\otimes h_{J_2}}h_{J_1^{(k)}}h_{J_1}\otimes h_{J_2}\\
&=\sum_{k=0}^{i_1}\sum_{J_2}\sum_{J_1}\beta_{J_1,k}|J_1^{(k)}|^{-1/2}\pair{\ave{b}_{J_2^{(i_2)}}}{h_{J_1^{(k)}}}_1\pair{\pair{f}{h_{J_2}}_2}{h_{J_1}}_1h_{J_1}\otimes h_{J_2}\\
&=:\sum_{k=0}^{i_1}\sum_{J_2}B_k(\ave{b}_{J_2^{(i_2)}},\pair{f}{h_{J_2}}_2)\otimes h_{J_2},
\end{split}
\]
where coefficients $\beta_{J_1,k}\in\{1,-1\}$ and the $L^2$ norm of the generalized paraproduct $B_k$ is uniformly bounded as mentioned before. Therefore, the same argument as for part I shows that
\[
\|II\|_{L^2(\mathbb{R}^n\times\mathbb{R}^m)}\lesssim (1+i_1)\|b\|_{\text{bmo}}\|f\|_{L^2(\mathbb{R}^n\times\mathbb{R}^m)},
\]
which completes the discussion of the second term, and thus proves that the mixed case is bounded.
\end{proof}

The upper bound result we just proved can be extended to $\mathbb{R}^{\vec{d}}$, to arbitrarily many parameters and an arbitrary number of iterates in the commutator. To do this, consider multi-parameter singular integral operators studied in \cite{O2}, which satisfy a weak boundedness property and are paraproduct free, meaning that any partial adjoint of $T$ is zero if acting on some tensor product of functions with one of the components being $1$. And consider a little product BMO function $b\in {\text{BMO}}_{\mathcal{I}}(\mathbb{R}^{\vec{d}})$. One can then prove 

\begin{theorem}\label{upperbd_Journe}
Let us consider $\mathbb{R}^{\vec{d}}$, $\vec{d}=(d_1,\ldots ,d_t)$ with a partition $\mathcal{I}=(I_s)_{1\le s \le l}$ of $\{1,\ldots ,t\}$ as discussed before. Let $b\in {\text{BMO}}_{\mathcal{I}}(\mathbb{R}^{\vec{d}})$ and let $T_s$ denote a multi-parameter paraproduct free Journ\'e operator acting on functions defined on $\bigotimes_{k\in I_s}\mathbb{R}^{d_k}$. Then we have the estimate below
\[
\|[T_1,\ldots[T_l,b]\ldots]\|_{L^2(\mathbb{R}^{\vec{d}})\righttoleftarrow}\lesssim \|b\|_{{\text{BMO}}_{\mathcal{I}}(\mathbb{R}^{\vec{d}})}.
\]
\end{theorem}

The part of the proof that targets the Journ\'e operators proceeds exactly the same as the bi-parameter case with the multi-parameter version of the representation theorem proven in \cite{O2}. 
Certainly, as the number of parameters increases, more mixed cases will appear. However, if one follows the corresponding argument above for each variable in each case, it is not hard to check that eventually, the boundedness of the arising paraproducts is implied exactly by the little product BMO norm of the symbol. The difficulty of higher iterates is overcome in observing that the commutator splits into commutators with no iterates, as was done in \cite{DO}. We omit the details.

The assumption that the operators be paraproduct free is sufficient for our lower estimate. The general case is currently under investigation by one of the authors. 
Important to our arguments for lower bounds with Riesz transforms is the corollary below, which follows from the control on the norm of the estimate in Theorem \ref{upperbd_Journe} by an application of triangle inequality. It is a stability result for characterizing families of Journ\'e operators.

\begin{corollary}\label{lemma_perturbation}
Let for every $1\le s \le l$ be given a collection $\mathcal{T}_s=\{T_{s,j_s}\}$ of paraproduct free Journ\'e operators on $\bigotimes_{k\in I_s}\mathbb{R}^{d_k}$ that characterize $BMO_{\mathcal{I}}(\mathbb{R}^{\vec{d}})$ via a two-sided commutator estimate
$$ \|b\|_{{\text{BMO}}_{\mathcal{I}}(\mathbb{R}^{\vec{d}})}   \lesssim \sup_{\vec{j}} \|[T_{1,j_1},\ldots[T_{l,j_l},b]\ldots]\|_{L^2(\mathbb{R}^{\vec{d}})\righttoleftarrow}\lesssim \|b\|_{{\text{BMO}}_{\mathcal{I}}(\mathbb{R}^{\vec{d}})}.$$
Then there exists $\varepsilon>0$ such that for any family of paraproduct free Journ\'e operators $\mathcal{T'}_s=\{T'_{s,j_s}\}$ with characterizing constants $\|T'_{s,j_s}\|_{CZ}\le \varepsilon$, the family $\{T_{s,j_s}+T'_{s,j_s}\}$ still characterizes $BMO_{\mathcal{I}}(\mathbb{R}^{\vec{d}})$.
\end{corollary}

\section{Weak Factorization}

It is well known, that theorems of this form have an equivalent formulation
in the language of weak factorization of Hardy spaces. We treat the model case $\mathbb{R}^{\vec{d}}=\mathbb{R}^{(d_1,d_2,d_3)}$ and $\text{BMO}_{(13)(2)}(\mathbb{R}^{\vec{d}})$ only for sake of simplicity. The other statements are an obvious generalization. For the corresponding collections of Riesz transforms $\mathcal{R}_{k,j_k}$  and $b\in \text{BMO}_{(13)(2)}(\mathbb{R}^{\vec{d}})$, $1\leq s\leq 3$, by unwinding the commutator one can define the operator $\Pi_{\vec{j}}$ such that
\[
\langle[R_{2,j_2},[R_{1,j_1}R_{3,j_3},b]]f,g\rangle_{L^2}=\langle b,\Pi_{\vec{j}}(f,g)\rangle_{L^2}.
\]

Consider the Banach space $L^2\ast L^2$ of all functions in $L^1(\mathbb{R}^{\vec{d}})$ of the form $f=\sum_{\vec{j}}\sum_i\Pi_{\vec{j}}(\phi_i^{\vec{j}},\psi_i^{\vec{j}})$ normed by
\[
\|f\|_{L^2\ast L^2}=\inf\{\sum_{\vec{j}}\sum_i\|\phi_i^{\vec{j}}\|_2\|\psi_i^{\vec{j}}\|_2\}
\]
with the infimum running over all possible decompositions of $f$. Applying a duality argument and the two-sided estimate in Corollary \ref{corollary_riesz} we are going to prove the following weak factorization theorem.

\begin{theorem}
$H^1_{\text{Re}}(\mathbb{R}^{(d_1,d_2)})\otimes L^1(\mathbb{R}^{d_3})+L^1(\mathbb{R}^{d_1})\otimes H^1_{\text{Re}}(\mathbb{R}^{(d_2,d_3)})$ coincides with the space $L^2\ast L^2$. In other words, for any $f\in H^1_{\text{Re}}(\mathbb{R}^{(d_1,d_2)})\otimes L^1(\mathbb{R}^{d_3})+L^1(\mathbb{R}^{d_1})\otimes H^1_{\text{Re}}(\mathbb{R}^{(d_2,d_3)})$ there exist sequences $\phi_i^{\vec{j}},\psi_i^{\vec{j}}\in L^2$ such that $f=\sum_{\vec{j}}\sum_i\Pi_{\vec{j}}(\phi_i^{\vec{j}},\psi_i^{\vec{j}})$ and $\|f\|\sim \sum_{\vec{j}}\sum_i\|\phi_i^{\vec{j}}\|_2\|\psi_i^{\vec{j}}\|_2$.
\end{theorem}

\begin{proof}
Let's first show that $L^2\ast L^2$ is a subspace of $H^1_{\text{Re}}(\mathbb{R}^{(d_1,d_2)})\otimes L^1(\mathbb{R}^{d_3})+L^1(\mathbb{R}^{d_1})\otimes H^1_{\text{Re}}(\mathbb{R}^{(d_2,d_3)})$. Recalling the remark after Theorem \ref{predual}, this is the same as to show $\forall f\in L^2\ast L^2$, $f$ is a bounded linear functional on $\text{BMO}_{(13)(2)}(\mathbb{R}^{\vec{d}})$. This follows from the upper bound on the commutators since 
\[
\langle b,\sum_{\vec{j}}\sum_i\Pi_{\vec{j}}(\phi_i^{\vec{j}},\psi_i^{\vec{j}})\rangle=\sum_{\vec{j}}\sum_i\langle[R_{2,j_2},[R_{1,j_1}R_{3,j_3},b]]\phi_i^{\vec{j}},\psi_i^{\vec{j}}\rangle.
\]
Now we are going to show
\[
\sup_{f\in L^2\ast L^2}\Big\{|\int fb|:\,\|f\|_{L^2\ast L^2}\leq1\Big\}\sim\|b\|_{\text{BMO}_{(13)(2)}}
\]
which gives the equivalence of $H^1_{\text{Re}}(\mathbb{R}^{(d_1,d_2)})\otimes L^1(\mathbb{R}^{d_3})+L^1(\mathbb{R}^{d_1})\otimes H^1_{\text{Re}}(\mathbb{R}^{(d_2,d_3)})$ norm and the $L^2\ast L^2$ norm, thus showing that the two spaces are the same.

To see this, note that the direction $\lesssim$ is trivial, and the direction $\gtrsim$ is implied by the lower bound of commutators. For any $b\in \text{BMO}_{(13)(2)}$, there exists $\vec{j}$ such that $\|b\|_{\text{BMO}_{(13)(2)}}\lesssim \|[R_{2,j_2},[R_{1,j_1}R_{3,j_3},b]]\|$. Hence, there exist $\phi,\psi\in L^2$ with norm $1$ such that
\[
\|b\|_{\text{BMO}_{(13)(2)}}\lesssim|\langle[R_{2,j_2},[R_{1,j_1}R_{3,j_3},b]]\phi,\psi\rangle|=|\langle b,\Pi_{\vec{j}}(\phi,\psi)\rangle|\leq LHS,
\]
which completes the proof.
\end{proof}

\section{Remarks about our results in $L^p$}\label{generalcase}

As mentioned before, the two-sided estimates stated in section \ref{section_riesz} and in particular Theorem \ref{theorem_riesz_choice_number} hold for all $1<p<\infty$. 
      The fact that upper estimates hold for the Riesz commutator in $L^p$ in the case where no tensor products are present is proved in \cite{LPPW} as well as \cite{LPPW2}. It stems from the fact that endpoint estimates for multi-parameter paraproducts hold for all $1<p<\infty$ \cite{MPTT1}, \cite{MPTT2}. This estimate carries over easily to tensor products of Riesz transforms or any other tensor products of operators for which we have $L^p$ estimates on the commutator: one uses $[T_{1}T_{2},b]=T_{1}[T_{2},b]+[T_{1},b]T_{2}$ to handle arising tensor products, followed by a correct use of the little product BMO norm. The argument is left as an exercise.

      The lower estimate or the necessity of the BMO condition can be derived from interpolation. In fact, suppose we have uniform boundedness of our commutators with operators running through all choices of Riesz transforms and some symbol $b$ in $L^p.$  Then  by duality, we have boundedness in $L^q$ where $1/p+1/q=1$. In fact, $[T,b]^*f=-[T^*,\bar{b}]f=-\overline{[\overline{T^*},b]\bar{f}}$ shows that the boundedness of adjoints is inherited. The same reasoning holds for iterated 
 commutators of tensor products. Thus by interpolation, the boundedness holds in $L^2$ and the symbol function $b$  necessarily belongs to the required BMO class.

%
%
%
%
%
%
%

\end{document}